\documentclass[11pt,reqno]{amsart}
\usepackage{amssymb,mathrsfs,bm,enumerate,comment}
\usepackage{amsfonts,dsfont,
  mathptmx,amsaddr,color}

\usepackage{txfonts}
\AtBeginDocument{\mathcode`v=\varv}
\AtBeginDocument{\mathcode`w=\varw}

\makeatletter

\newtheorem{theorem}{Theorem}[section]
\newtheorem{lemma}[theorem]{Lemma}
\newtheorem{corollary}[theorem]{Corollary}

\theoremstyle{definition} \newtheorem{remark}[theorem]{Remark}
\newtheorem{remarks}[theorem]{Remarks}
\newtheorem{definition}[theorem]{Definition}

\allowdisplaybreaks \numberwithin{equation}{section}

\newcommand{\eps}{\varepsilon} \newcommand\1{\hbox{\kern.375em\vrule
    height1.57ex depth-.1ex width.05em\kern-.375em \rm 1}}

 \newcommand\E{\mathbb{E}}
 
\newcommand\N{\mathbb{N}} \newcommand\R{\mathbb{R}}
\renewcommand\P{\mathbb{P}}

\newcommand\SF{\mathscr F}

\newcommand\mequal{\overset{\text{\tiny m}}{=}}

\newcommand\HS{\text{\rm\tiny HS}}

\def\mathpal#1{\mathop{\mathchoice{\text{\rm #1}}%
    {\text{\rm #1}}{\text{\rm #1}}%
    {\text{\rm #1}}}\nolimits} \def\id{{\mathpal{id}}}

\def\beq{\begin{equation}}\def\si{\sigma} \def\d{{\rm d}}
   \def\vd{\mathrm{d}} \def\bd{{\bf d}}
  \def\r{\right} \def\l{\left} \def\e{\operatorname{e}}\def\<{\langle}
  \def\>{\rangle} 
   \def\Ent{{\operatorname{Ent}}}
  \def\Ric{{\operatorname{Ric}}} 
   
  \def\Hess{{\operatorname{Hess}}} \def\tr{\operatorname{tr}}
  \newcommand{\ptr}{/\!/} \def\ff{\frac}\def\nn{\nabla}
  \def\II{{\operatorname{II}}} \def\pp{\partial}\def\DD{\Delta}
  \def\B{\mathcal B}\def\ll{\lambda}
  \def\vv{\varepsilon}\def\beq{\begin{equation}}\def\beg{\begin}
      \def\rr{\rho} \def\ss{\sqrt}\def\aa{\alpha}\def\tt{\tilde}

      \newcommand\newdot{{\kern.8pt\cdot\kern.8pt}}

      \makeatother

\begin{document}

      \title[Second order Bismut formulae for Neumann semigroups]
      {Second Order Bismut formulae and applications to Neumann
        semigroups on manifolds}

      \author{Li-Juan Cheng\textsuperscript{1}, Anton
        Thalmaier\textsuperscript{2} and Feng-Yu
        Wang\textsuperscript{3}}

      \address{\textsuperscript{\rm1}School of  Mathematics, Hangzhou Normal University,\\
        Hangzhou 311123, The People's Republic of China}
      \address{\textsuperscript{\rm2}Mathematics Research Unit, FSTC, University of Luxembourg,\\
        Maison du Nombre, 4364 Esch-sur-Alzette, Luxembourg}
      \address{\textsuperscript{\rm3}Center for Applied Mathematics,
        Tianjin
        University,\\ Tianjin 300072, People's Republic of China\\ {\rm and} Department of Mathematics, Swansea University,\\
        Swansea SA1 8EN, United Kingdom}
      \email{lijuan.cheng@hznu.edu.cn, anton.thalmaier@uni.lu,
        wangfy@tju.edu.cn}

      \begin{abstract} Let $M$ be a complete connected Riemannian
        manifold with boundary $\pp M$, and let $P_t$ be the Neumann
        semigroup generated by $\ff 1 2 L$ where $L=\DD+Z$ for a
        $C^1$-vector field $Z$ on $M$. We establish Bismut type
        formulae for $LP_t f$ and $\Hess_{P_tf}$ and present estimates
        of these quantities under suitable curvature conditions.  In
        case when $P_t$ is symmetric in $L^2(\mu)$ for some
        probability measure $\mu$, a new type of log-Sobolev
        inequality is established which links the relative entropy
        $H$, the Stein discrepancy~$S$, and relative Fisher
        information~$I$, generalizing the corresponding result of
        \cite{CTW} in the case without boundary.
      \end{abstract}

       \date{\today}

      \maketitle

      \section{Introduction}

      Consider a $d$-dimensional complete Riemannian manifold $M$,
      possibly with non-empty boundary $\partial M$, and let $X_t$ be
      the reflecting diffusion process on $M$ generated by $\frac12 L$
      where $L=\Delta+Z$; here $\Delta$ is the Laplace-Beltrami
      operator and $Z$ a smooth vector field on
      $M$. 
      According to \cite{Hsu,Ikeda-Watanabe:1989,Wbook14}, the
      reflecting diffusion process $X_t^x$ starting at $x$ can be
      constructed as solution to the following SDE on $M$ with
      reflection:
      \begin{equation}\label{Eq:reflBM}
        \vd X^x_t=\ptr_t\circ\,\vd B_t
        +\frac12 Z(X^x_t)\,\vd t+\frac12 N(X^x_t)\,\vd l_t^x,\quad X_0^x=x,
      \end{equation}  
      where $B_t$ is a standard Brownian motion on the tangent space
      $T_xM\equiv\R^d$, $\ptr_t\colon T_xM\rightarrow T_{X^x_t}M$ the
      stochastic parallel transport along $X_t^x$, $N$ the inward
      normal unit vector field on $\partial M$, and $l_t^x$ the local
      time of $X_t^x$ on $\partial M$.  Throughout this paper, we
      assume that SDE \eqref{Eq:reflBM} is non-explosive. Then the
      Neumann semigroup $P_t$ generated by $\ff 1 2 L$ is given by
$$P_tf(x)=\E\l[f(X_t^x)\r],\quad t\ge 0,\ x\in M,\ f\in \B_b(M)$$
where $\B_b(M)$ denotes the set of bounded measurable functions on
$M$.
 
To study the regularity of diffusion semigroups using tools from
stochastic analysis, Bismut \cite{Bismut:1984} introduced his famous
probabilistic formula for the gradient of heat semigroups on
Riemannian manifolds without boundary. This type of formulae has been
studied in \cite{EL94, Tha97, DTh2001} using martingale arguments, and
been extended to second order derivatives in \cite{APT2003, EL94,
  St00, StT98, Tha97, Th19, Li2021}.

In the case the boundary of $M$ is non-empty, Bismut type formulae
have been derived in \cite{Wbook14, Cheng_Thalmaier_Thompson:2018} for
the gradient of the Neumann semigroup $P_t$, see also \cite{Qian:1997,
  Hsu, Wang:2020} for gradient estimates. In this paper, we aim at
establishing Bismut type formulae for second order derivatives of the
Neumann semigroup, along with some geometric applications.

Let $\Ric_Z:=\Ric-\nabla Z$ where $\Ric$ is the Ricci curvature
tensor, and let $\II$ be the second fundamental form of the boundary:
\begin{align*}
  \II(X,Y)=-\l<\nabla_X N, Y\r>,\quad \text{$X,Y\in T_x\partial M,\ x\in\partial M$}.
\end{align*}
A derivative formula for $P_tf$ is given in \cite{Hsu,Wbook14} by
constructing an appropriate multiplicative functional. Throughout the
paper, we assume that the reflecting diffusion process generated by
$L$ is non-explosive, and that there exist functions $K\in C(M)$ and
$\si \in C(\partial M)$ such that \beq\label{RCC} \Ric_Z:=\Ric-\nabla
Z\geq K,\quad \ \II\geq \sigma,\end{equation} i.e.
$\Ric_Z(X,X)\ge K(x)|X|^2$ for $x\in M,\ X\in T_xM$, and
$\II(X,X)\ge \si(x) |X|^2$ for $x\in \pp M,\ X\in T_x \pp M.$ Under
the assumption that

\smallskip\fbox{  \parbox{0.8 \textwidth}{
    \beg{enumerate} \item[{\bf (A)}]  the functions $K$ and $\si$ in \eqref{RCC} are constant,  $\E[\e^{\si^- l_t}]<\infty$ for any $t\geq0$,\end{enumerate} }
}\smallskip

\noindent
a Bismut type
formula for $\nabla P_tf$ has been established in \cite{Wbook14} for $f\in \B_b(M)$
such that  $\nabla P_{\cdot}f$ is bounded on $[0,t]\times M$. More precisely, there exists
a family of random homomorphisms
$Q_t\colon T_xM\rightarrow T_{X_t^x}M$ with the property that
\begin{align*}
  |Q_t|\leq \e^{-Kt/2-\sigma l_t/2}\ \text{ and }\ 
  \langle N(X_t^x), Q_t(v)\rangle\1_{\{X_t^x\in \partial M\}}=0,\quad v\in T_xM,
\end{align*}
such that
\begin{align}\label{Eq:gradient_formula}
  \nabla P_tf(v)=\E\left[\langle\nabla f(X_t^x), Q_t(v)\rangle\right],\quad v\in T_xM,
\end{align}
and
\begin{align}\label{Eq:BismutdPtf}
  \nabla P_tf(v)=\E\l[f(X_t^x)\int_0^t\langle h'(s) Q_s(v), \ptr_s \vd B_s \rangle\r],
  \quad v\in T_xM
\end{align}
for any choice of a non-negative $h\in C_b^1([0,t])$ such that
$h(0)=0,\ h(t)=1$.
When $\Ric_Z$ and $\II$ are bounded from below, the second part of condition {\bf (A)} holds if either $\pp M$ is  convex, or if $\pp M$ has strictly positive injectivity radius,   the sectional curvature of $M$ being bounded above and $Z$ bounded, see \cite[Section 3.2]{Wbook14}.   

\

The aim of this paper is to extend  \eqref{Eq:gradient_formula} and
\eqref{Eq:BismutdPtf} to second order derivatives and to establish Bismut type formulae for $L P_tf$ and $\Hess_{P_t f}:=\nabla \bd P_tf$, along with some applications. When compared to the case without boundary as in \cite{APT2003},  the present study faces an essential new difficulty. Indeed, by formal calculations, the Bismut formula for second derivatives of $P_tf$ includes 
a stochastic integral of $Q_t^{-1}$, the inverse of  the above mentioned multiplicative functional $Q_t$. However, in the present setting $Q_t$ is singular near the boundary
so that existence of the desired stochastic integral poses a problem.    

\

In Sections~\ref{LPf},  we derive a Bismut type formula for $LP_tf$ also in terms of the multiplicative functional $Q_t$, which provides as consequence an upper estimate depending on the lower bounds of $\Ric_Z$ and~$\II$, more precisely, condition ${\bf(A)}$ and $\|Z\|_{\infty}<\infty$. 
This estimate is new even in the case  without boundary where the existing estimate in \cite{Th19} depends on the  uniform norm of $\Ric_Z$.  
 
\

In Section~\ref{HessPf},  we  establish a   Bismut type  
formula for $\Hess_{P_tf}$ and use it for Hessian estimates of $P_tf$. Establishing formulas for the Hessian naturally requires more knowledge on the curvature of the manifold. Note that in  \cite{APT2003,Th19, Li2021, StT98}, the authors used information related to curvature and
its derivative to establish 
a Hessian formula and deduced estimates in terms of these curvature conditions. When it comes to manifolds with boundary, it seems unavoidable to exploit geometric information concerning the boundary as well.  Before going into the details, let us remark that the multiplicative functional $Q_t$ in the derivative 
formula  \eqref{Eq:gradient_formula} satisfies $$\langle N(X_t), Q_t(v) \rangle\1_{\{X_t\in \partial M\}}=0$$  which is reasonable since 
$$\langle \nabla P_{T-t}f(X_t), \ N(X_t) \rangle\1_{\{X_t\in \partial M\}}=0.$$
It follows that  to express  $\nabla P_tf$ on the boundary, information on the second fundamental
form
$$\II^{\sharp} (P_{\partial} (v))=-(\nabla _{P_{\partial} (v)}N)^{\sharp}$$ 
is sufficient. However, when it comes to the second order derivative of $P_tf$ on the boundary, no condition like 
$$\Hess_{ P_{T-t}f} (N(X_t),\cdot) \1_{\{X_t\in \partial M\}}=0$$
is satisfied, 
which naturally demands for full information on $\nabla N$. This indicates that one not only needs to control $\II$ but also $\nabla_N N$.  For this reason, in Section~\ref{HessPf}, two new 
functionals $\tilde{Q}_t$ and  $W_t$ are introduced in
\eqref{tilde-Q} and \eqref{W-def} respectively, which our Bismut formula for $\Hess_{ P_tf}$ will be based on and which then allow to derive upper bounds.
    
\

In Section \ref{Stein-log-Sobolev},  we apply the Hessian estimates of $P_tf$
to prove inequalities connecting  the relative entropy $H$, the  Stein
discrepancy $S$,  and the  relative Fisher information $I$, which extend the corresponding results derived in our recent work  \cite{CTW} for $\pp M=\emptyset$ 
to the case with boundary; see Ledoux, Nourdin and Peccati \cite{LNP15}
for the earlier study in the Euclidean case  $M=\R^d$.

\section{Bismut   formula and   estimate for  $LP_tf$ }\label{LPf}

To state the main result, we first recall the construction of the multiplicative functional 
$Q_t$ appearing in the Bismut formula, see \cite{Th19} for the case without boundary.

For $t\ge 0$, let $\ptr_{0\to t}: T_{X_0}M\to T_{X_t}M$ denote stochastic parallel transport along the paths of the reflecting diffusion process $X$.  
The covariant differential  ${\rm D}$ in $t\ge 0$ is defined as 
${\rm D}   :=    \ptr_{0\to t}\,\d\, \ptr_{t\to 0}$ where  $\d$ is the usual It\^o stochastic differential in $t\ge 0$. For a process $v_t\in T_{X_t}M$ we then have
$${\rm D} v_t = \ptr_{0\to t}\,\d\, \ptr_{t\to 0}\,v_t,\quad t\ge 0.$$
For $n\in \N$ and $t\geq 0$, let
$Q^{(n)}_{t}\colon T_{X_0}M\rightarrow T_{X_t}M$ solve the 
covariant differential equation:
\beq\label{QN} 
{\rm D} Q_{t}^{(n)}=-\frac12 \l\{\Ric_Z^{\sharp}(Q_{t}^{(n)})\, \vd t+\II^{\sharp}(Q_{t}^{(n)})\, \vd l_t+nP_N(Q_{t}^{(n)})\, \vd l_t\r\},\quad t\geq 0, \ \ Q_{0}^{(n)}=\id,
\end{equation}
where $ \id$ is the identity map on $T_{X_0}M$ and $P_N$ the projection operator onto the normal direction $N$ of $\pp M$ such that 
when $X_t\in \pp D$, 
  $$P_NQ_{t}^{(n)}v=\langle Q_{t}^{(n)}v, N(X_t) \rangle\, N(X_t),\quad v\in T_{X_0}M.$$
Furthermore for $P_\pp: T_xM\to T_x\pp M$ being the projection operator for $x\in \pp M$, let
  $$\<\II^{\sharp}(Q_{t}^{(n)})v_1, v_2\>:= \II( P_\pp Q_{t}^{(n)} v_1, P_\pp v_2),\quad v_1,v_2\in T_{X_t}M,\  X_t\in \pp M.$$
By the curvature conditions \eqref{RCC}, we then have 
 \beq\label{QNU} \sup_{n\ge 1} |Q_t^{(n)}|\le \e^{-\ff 1 2 \int_0^t K(X_s)\,\d s-\ff 1 2 \int_0^t\si(X_s)\,\d l_s}, \quad t\ge 0,\end{equation} 
  \beq\label{QNU'} \int_0^t|P_N (Q_s^{(n)})|^2\, \d l_s \le    \ff 1 n \int_0^t |Q_s^{(n)}|^2 \big\{K^-(X_s)\,\d s + \si^-(X_s)\,  \d l_s\big\}\to 0\ \text{ as}\ n\to\infty.  \end{equation}
 Define 
  \beq\label{An} \{Q_t^{(n)}\}^{-1}:= \ptr_{t\to 0}\{Q_t^{(n)} \ptr_{t\to 0}\}^{-1}:\ T_{X_t}M\to T_xM\end{equation} 
  where $\{Q_t^{(n)} \ptr_{t\to 0}\}^{-1}$ is the inverse of the operator
  $$Q_t^{(n)} \ptr_{t\to 0}:\ T_{X_t}M\to T_{X_t}M.$$  
  To show that $\{Q_t^{(n)}\}^{-1}$ exists, let
$$ \tau_k:=\inf\big\{t\ge 0:\ \rr_o(X_t)\ge k\big\},\ \ k\ge 1,$$
 where $\rr_o$ is the Riemannian distance to a reference point $o\in M$. Fix $T>0$. By \cite[Lemma 3.1.2]{Wbook14}, we have
\beq\label{TK} \E[\e^{\ll l_{T\land\tau_k} }]<\infty,\ \ \ll>0.\end{equation}
Since $\Ric_Z$ and $\II$ are locally bounded, \eqref{QN} and \eqref{TK} imply that $Q_t^{(n)} \ptr_{t\to 0}$ is invertible with
\beq\label{QNT} \E \bigg[\sup_{t\in [0, T\land\tau_k]} \big|\{Q_t^{(n)}\}^{-1}\big|^p\ \bigg]= \E \bigg[\sup_{t\in [0, T\land\tau_k]} \big|\{Q_t^{(n)}\ptr_{t\to 0}\}^{-1}\big|^p \bigg] <\infty,\ \ p,k\ge 1.\end{equation}

 To derive a Bismut formula for $LP_tf$, we need to estimate the martingales 
\beq\label{MN} M_t^{(h,n)} := \int_0^{t}\bigg\< h_s Q_s^{(n)} \int_0^s h_r \{Q_{r}^{(n)}\}^{-1}\ptr_r \vd B_r,\,  \ptr_s\vd B_s\bigg\>,\ \ n\ge 1\end{equation}
 for a reference  adapted real process $h$.  When   $M$ is compact,   \cite[Theorem 3.4]{Hsu} implies that 
as  $n\rightarrow \infty$  the process $Q_{t}^{(n)}$ converges in $L^2(\P)$ to an
  adapted right-continuous process $Q_{t}$ with left-limits such that
  $P_{N}Q_{t}=0$ if $X_t\in \partial M$.  This construction has been extended  in
  \cite[Proof of Theorem 3.2.1]{Wbook14} to non-compact manifolds. However, although $\{Q_t^{(n)}\}^{-1}$ exists for every $n\ge 1$,
  $Q_t$ is not invertible on the boundary since $P_{N}Q_{t}=0.$ Hence a priori,
  existence of the stochastic integral
$$  \int_0^{t}\bigg\< h_s Q_{s} \int_0^s h_r Q_{r}^{-1}\ptr_r \vd B_r,\,  \ptr_s\vd B_s\bigg\>$$
 is not obvious.

  \beg{lemma}\label{L1}  Let $K\in C(M)$ and $\sigma \in C(\partial M)$ such that $\eqref{RCC}$ holds.  Then  for any 
    adapted real process $(h_t)_{t\in [0,T]}$ with
   \begin{align}\label{def-Ch}
   C(h):=\E\l[\int_0^T h_s^2 \e^{\int_0^s K^-(X_r)\d r +\int_0^s \si^-(X_r) \d l_r}\d s \r]<\infty,
   \end{align}
the martingales $M_t^{(h,n)}$  in $\eqref{MN}$ satisfy
 \beq\label{MES}  \beg{split}\sup_{n\ge 1} \E\l[\sup_{t\in [0,T]} \big|M_t^{(h,n)}\big|\r]
\le  3\Big(3+\sqrt{10}\Big)\,  \left(C(h)\, \E\int_0^T h_s^2\, \vd s\right)^{1/2}<\infty.
\end{split} \end{equation}
In addition, if there is a constant $\aa>1$ such that
\beq\label{JQ}\E\left[\bigg(\int_0^T h_s^2  \d s\bigg)^{\aa}\right] <\infty,\end{equation}
then there exists a real random  variable $M_T^h$ with 
$\E \big[|M_T^h|^{\ff{2\aa}{1+\aa}}\big] <\infty$,
and a subsequence $n_m\to\infty$ as $m\to\infty$, such that  
\beq\label{WK} \lim_{m\to\infty }\E\big[\eta M_T^{(h,n_m)}\big]= \E\big[\eta M_T^h\big],\ \ \ \eta \in L^{\ff{2\aa}{\aa-1}}(\P).\end{equation} 
In  case \eqref{JQ} holds for $\alpha=1$ as well, one has
$$\E\big[|M_T^h|\big]\le 3\Big(3+\sqrt{10}\Big)\,  \left(C(h)\, \E\bigg[ \int_0^T h_s^2\, \vd s\bigg]\right)^{1/2}.$$
   \end{lemma}
  \beg{proof}    (a) We first prove \eqref{MES}. By Fatou's lemma, it suffices to  show
  \beq\label{PO1}\beg{split} I_{k,n}&:=\E\left[\sup_{t\in [0,T\land\tau_k]} \bigg|\int_0^{t}\bigg\langle h_s Q^{(n)}_{s} \int_0^s h_r \{Q^{(n)}_{r}\}^{-1}\ptr_r \vd B_r,\,  \ptr_s\vd B_s\bigg\rangle\bigg|\right] \\
&\le  3\l(3+\sqrt{10}\r)\, \sqrt{C(h)}\,  \left( \E\int_0^T h_s^2\, \vd s\right)^{1/2},\ \ \ k,n\ge 1.\end{split}\end{equation} 
For fixed $n\ge 1$ let 
\beg{align*}  \xi_s := Q_s^{(n)} \int_0^s h_r \{Q_r^{(n)}\}^{-1} \ptr_r\d B_r,\ \ s\ge 0.\end{align*}  
By Lenglart's inequality (see \cite[Proposition 5.69]{Bbook14}) and Schwartz's inequality, it follows that
\beq\label{PO2} \beg{split} I_{k,n}& \le 3 \E\bigg[\bigg(\int_0^{T\land\tau_k} |\xi_s|^2h_s^2 \d s\bigg)^{1/2}\bigg]\\
&\le 3 \E\left[\sup_{s\in [0,T\land\tau_k]} |\xi_s| \e^{-\ff 1 2 \int_0^s  K^-(X_r)\d r- \ff 1 2 \int_0^s \si^-(X_r)\d l_r} \bigg(\int_0^{T\land\tau_k} h_s^2 \e^{\int_0^s K^-(X_r)\d r+\int_0^s \si^-(X_r)\d l_r} \d s\bigg)^{1/2}\right]\\
&\le 3\left\{\E\l[\sup_{s\in [0,T\land\tau_k]} |\xi_s|^2 \e^{- \int_0^s  K^-(X_r)\d r-   \int_0^s \si^-(X_r)\d l_r} \r] C(h)\right\}^{1/2}.\end{split}\end{equation}  
Furthermore, by It\^o's formula we have
  \beg{align*} \d |\xi_s|^2 &= 2\<\xi_s, h_s\ptr_s\d B_s\>-\big\{\Ric_Z(\xi_s,\xi_s)\d s+ \II(\xi_s,\xi_s)\d l_s\big\}- n |P_N \xi_s|^2\d l_s+\d[\xi,\xi]_s\\
  &\le 2\<\xi_s, h_s\ptr_s\d B_s\>-K(X_s)|\xi_s|^2\d s-\si(X_s)|\xi_s|^2\d l_s+h_s^2\d s,\ \ 0\le s< \tau_k, \end{align*}
which implies
\beq\label{PM} \d \Big\{|\xi_s|^2  \e^{- \int_0^s  K^-(X_r)\d r-   \int_0^s \si^-(X_r)\d l_r}\Big\}\le    \e^{- \int_0^s  K^-(X_r)\d r-   \int_0^s \si^-(X_r)\d l_r}\big\{2\<\xi_s, h_s\ptr_s\d B_s\>+h_s^2\d s\big\},\quad 0\le s< \tau_k.\end{equation} 
By the condition $\xi_0=0$ and Lenglart's inequality,  we  have
\beg{align*} &\E\l[\sup_{s\in [0,T\land\tau_k]} |\xi_s|^2 \e^{- \int_0^s  K^-(X_r)\,\d r-\int_0^s \si^-(X_r)\,\d l_r} \r]\\
&\le 6 \E\l[\bigg(\int_0^{T\land\tau_k}|\xi_s|^2h_s^2 \e^{- 2\int_0^s  K^-(X_r)\,\d r-  2 \int_0^s \si^-(X_r)\,\d l_r}\,\d s\bigg)^{1/2}\r] +\E \bigg[\int_0^T h_s^2\,\d s\bigg]\\
&\le   6 \E\bigg[\bigg(\sup_{s\in [0,T\land\tau_k]} |\xi_s|^2 \e^{- \int_0^s  K^-(X_r)\,\d r-   \int_0^s \si^-(X_r)\,\d l_r} \bigg) \bigg(\int_0^{T\land\tau_k}h_s^2  \,\d s\bigg)^{1/2}\bigg] +\E \bigg[\int_0^T h_s^2\,\d s \bigg] \\
&\le   \ff 1 {2\delta} \E\bigg[\sup_{s\in [0,T\land\tau_k]} |\xi_s|^2 \e^{- \int_0^s  K^-(X_r)\,\d r-   \int_0^s \si^-(X_r)\,\d l_r} \bigg]+18 \delta \E\int_0^T h_s^2\,\d s+\E \bigg[\int_0^T h_s^2\,\d s\bigg],\end{align*}
for any $\delta>0$. Taking the optimal choice $\delta= \ff 1 6\big(3+\ss{10}\big)$, we obtain 
$$\E\l[\sup_{s\in [0,T\land\tau_k]} |\xi_s|^2 \e^{- \int_0^s  K^-(X_r)\,\d r-   \int_0^s \si^-(X_r)\,\d l_r} \r]\le \Big(3+\sqrt{10}\Big)^2 \E \bigg[ \int_0^T h_s^2\,\d s\bigg].$$
Combining this with \eqref{PO2}, estimate \eqref{PO1} follows by letting $k$ tend to $\infty$.
\smallskip

(b) Assume that \eqref{JQ} holds for some $\aa>1$. By estimate \eqref{PM} and the Burkholder-Davis-Gundy inequality, we can find  constants $c_1,c_2>0$ such that 
\beg{align*} &\E\l[\sup_{s\in [0,T\land\tau_k]}\Big( |\xi_s|^2 \e^{- \int_0^s  K^-(X_r)\,\d r-   \int_0^s \si^-(X_r)\,\d l_r}\Big)^\aa \r]\\
&\le   c_1  \E\l[\l(\int_0^{T\land\tau_k}|\xi_s|^2h_s^2 \e^{- 2\int_0^s  K^-(X_r)\,\d r-  2 \int_0^s \si^-(X_r)\,\d l_r}\, \d s\r)^{\aa/ 2}\r] +c_1 \E\l[\bigg(\int_0^T h_s^2\,\d s\bigg)^\aa\r]\\
&\le   c_1 \E\l[\l(\l[\sup_{s\in [0,T\land\tau_k]} |\xi_s|^2 \e^{- \int_0^s  K^-(X_r)\,\d r-   \int_0^s \si^-(X_r)\,\d l_r} \r] \int_0^{T\land\tau_k}h_s^2  \,\d s\r)^{\aa/ 2}\r] +c_1 \E\l[\bigg(\int_0^T h_s^2\,\d s\bigg)^\aa\r] \\
&\le   \ff {1} {2} \E\l[\sup_{s\in [0,T\land\tau_k]}\Big( |\xi_s|^2 \e^{- \int_0^s  K^-(X_r)\,\d r-   \int_0^s \si^-(X_r)\,\d l_r} \Big)^\aa\r]+\ff {c_1^2} {2}\E \l ( \int_0^T h_s^2\,\d s\r)^{\aa}+c_2 \E \bigg(\int_0^T h_s^2\,\d s\bigg)^\aa.\end{align*}
Together with \eqref{JQ} this implies 
$$ G:=  \E\l[\sup_{s\in [0,T]}\bigg( |\xi_s|^2 \e^{- \int_0^s  K^-(X_r)\,\d r-   \int_0^s \si^-(X_r)\,\d l_r}\bigg)^\aa \r]<\infty.$$
On the other hand, by Burkholder-Davis-Gundy's inequality, there exists a constant $c_3>0$ such that
\beg{align*} 
\E\l[\big|M_T^{(h,n)}\big|^{\ff {2\aa}{1+\aa}}\r]&\le c_3 \E\l[\bigg(\int_0^T |\xi_s|^2 \e^{- \int_0^s  K^-(X_r)\,\d r-   \int_0^s \si^-(X_r)\,\d l_r}h_s^2\e^{ \int_0^s  K^-(X_r)\,\d r+ \int_0^s \si^-(X_r)\,\d l_r}\,\d s\bigg)^{\ff \aa {1+\aa}}\r]\\
&\le c_3 {G}^{\ff 1 {1+\alpha}}\l(\E  \int_0^T h_s^2\e^{ \int_0^s  K^-(X_r)\,\d r+ \int_0^s \si^-(X_r)\,\d l_r}\,\d s\r)^{\ff \aa{1+\aa}}<\infty,\ \ n\ge 1.\end{align*}
Thus $\{M_T^{(h,n)}\}_{n\ge 1}$ is bounded in $L^{\ff{2\aa}{1+\aa}}(\P)$, and hence has a subsequence converging weakly to a random variable $M_T^h$ in $L^{\ff{2\aa}{\aa-1}}(\P)$.
    \end{proof}

\begin{theorem}\label{Hessian-th}  
 Let $K\in C(M)$ and $\si\in C(\partial M)$ such that $\eqref{RCC}$ holds.  
 For $T>0$ and  $x\in M$, let $h_t$ be an adapted real process such that
 $\int_0^T h_s\,\d s=-1$ and 
 \beq\label{CO} 
\E\l[\int_0^T h_s^2 \l (\e^{\int_0^s K^-(X^x_r)\,\d r +\int_0^s \si^-(X^x_r) \,\d l_r}+|Z(X_s^x)|^2\r)\,\d s\r] <\infty.
  \end{equation}
 Then, for any $f\in \B_b(M)$,  
  \beq\label{BSL} 
    L(P_Tf)(x)=2\,\E\l[f(X_T^x)\bigg(M_T^{ h}+ \int_0^{T}\Big\langle \tt h_s  h_sZ(X_s^x),\,  \ptr_s\vd B_s\Big\rangle\bigg) \r],
  \end{equation} 
where $\tt h_t:=1+\int_0^t h_s\, \d s$.
Consequently, 
  \beq\label{BL0} | L(P_Tf)(x)|\le 6 \|f\|_\infty \l\{ \E\bigg[\bigg(\int_0^T |\tt h_s h_s Z(X_s^x)|^2\,\d s\bigg)^{1/2}\bigg] +  \Big(3+\ss {10}\Big)  \bigg(C(h)\, \E\int_0^T h_s^2\, \vd s\bigg)^{1/2}\r\}.\end{equation}
 
 \end{theorem}  
%

\begin{proof} (1)
  We first assume 
  $f\in \mathcal{C}_{N}^{\infty}(L)$, the class of functions $f\in C^\infty(M)$ such that $Nf|_{\pp M}=0$ and  $\|Lf\|_\infty<\infty.$  
    In this case, we have the Kolmogorov  equations  (see \cite[Theorem 3.1.3]{Wbook14}),
\beq\label{BC} \pp_t P_{T-t} f = - L P_{T-t} f= - P_{T-t} Lf,\ \ N P_tf|_{\pp M}=0,\ \ t\in [0,T].\end{equation}  
In the sequel, we write for simplicity
 $$X_t=X_t^x, \ \ N_t=N(X_t),\ \   Z_t=Z(X_t),\ \    M_t:=LP_{T-t}f(X_t),\quad t\in [0, T].$$
Furthermore, we write $A_t \mequal B_t$ for two processes  $A_t$ and $B_t$ if the difference $A_t-B_t$ is a local martingale.  
  By It\^{o}'s formula and  \eqref{BC},   we obtain 
  \begin{align*}
    \vd M_t&=\langle \nabla (LP_{T-t}f)(X_t), \ptr_t\, \vd B_t \rangle +\partial_t(LP_{T-t}f)(X_t)\, \vd t+\frac12 L(LP_{T-t}f)(X_t)\, \vd t+\frac12 N(LP_{T-t}f)(X_t)\, \vd l_t\\
           &=\langle \nabla (LP_{T-t}f)(X_t), \ptr_t\, \vd B_t \rangle+\frac12 N(P_{T-t}Lf)(X_t)\, \vd l_t
           =\langle \nabla (LP_{T-t}f)(X_t), \ptr_t\, \vd B_t\rangle,\ \ \ t\in [0,T].
  \end{align*}
Then 
$$\d (M_t \tt h_t^2)=\tt h_t^2 \,\d M_t+ 2 \tt h_t (\tt h_t)' M_t\,\d t=\tt h_t^2\,\d M_t+ 2 \tt h_t h_t M_t\,\d t,$$
 which together with $\tt h_0=1$ implies 
 \begin{align}\label{L-eq1}
   (LP_{T-t}f)(X_t)\,\tt h_t^2- LP_Tf(x) \mequal 2\int_0^tLP_{T-s}f(X_s)\tt h_sh_s\vd s. 
  \end{align} 
With $\Delta=-\bd^*\bd$ and $L=\Delta+Z$, we have 
   \begin{align}
   - (LP_{T-t}f)(X_t) 
    =\l\{ \bd^*(\bd P_{T-t}f)-(\bd P_{T-t}f)(Z)\r\}(X_t).\label{L-eq2}
  \end{align}
 Combined with \eqref{L-eq1} this further yields
  \begin{align}\label{LPf-formula1}
    ( L P_{T-t}f)(X_t) \tt h_t^2-LP_{T}f(x)
 & \overset{\rm m}{=} 2\int_0^t LP_{T-s}f (X_s) \tt h_s h_s\, \vd s \notag \\
 &=-2 \int_0^t \bd^* (\bd P_{T-s}f) \tt h_s h_s\, \vd s +2 \int_0^t (\bd P_{T-s}f)(Z) \tt h_s h_s\, \vd s.
  \end{align}
  
    Let $Q_t^{(n)*}\colon T_{X_t}M\to T_xM $ be the adjoint operator to $Q_t^{(n)}$.
 By It\^o's formula and the Weitzenb\"{o}ck formula, we obtain 
  \begin{align*}
    \vd \l\{(\bd P_{T-t}f)(X_t)Q_t^{(n)}\r\}=\big(\nabla_{\ptr_t\vd B_t}\bd P_{T-t}f(X_t)\big)(Q^{(n)}_t)
    +\l\{\Hess_{P_{T-t}f}(N_t,N_t)\,Q_t^{(n)*} N_t\r\}(X_t) \, \vd l_t.
  \end{align*}
  Combining this with 
 $$ \bd ^*(\bd P_{T-t}f)\tt h_th_t\,\vd t
    =-\l\{(\nabla_{\ptr_t\, \vd B_t}\bd P_{T-t}f)(Q^{(n)}_t\tt h_th_t(Q^{(n)}_t)^{-1}\ptr_t\,\vd B_t)\r\}(X_t),$$
and using It\^o's formula, we derive    
  \begin{align}
    \int_0^t\bd ^*(\bd P_{T-s}f)\tt h_sh_s\,\vd s
    &\mequal- \big(\bd P_{T-t}f\big)\bigg(Q^{(n)}_t\tt h_t\int_0^th_s(Q^{(n)}_s)^{-1}\ptr_s\,\vd B_s\bigg) \notag\\
    &\quad+\int_0^t\Hess _{P_{T-s}f}(N_s,N_s)\Big\<N_s,\   \tt h_sQ^{(n)}_{s}\int_0^sh_r(Q^{(n)}_{r})^{-1} \ptr_r\,\vd B_r\Big\>\,\vd l_s \notag\\
    &\quad +\int_0^t (\bd P_{T-s}f)\Big(h_s Q^{(n)}_{s} \int_0^s h_r (Q^{(n)}_{r})^{-1}\ptr_r\, \vd B_r\Big) \, \vd s. \label{L-formula2}
  \end{align}
  To deal with the last term of the above equation and the second term on the right-hand side of  \eqref{LPf-formula1}, we observe that 
   by It\^o's formula 
  $$\d P_{T-s}f(X_s)= \<\nn P_{T-s}f(X_s), \ \ptr_s\, \vd B_s\>= (\bd  P_{T-s}f)(\ptr_s\, \vd B_s),$$ so that
  \begin{align*}
    &\int_0^t\big(\bd P_{T-s}f\big)(Z_s)\tt h_sh_s\, \vd s \mequal P_{T-t}f(X_t)
      \int_0^t\langle \tt h_s h_sZ_s,\ \ptr_s\vd B_s\rangle,\\
    & \int_0^t (\bd P_{T-s}f)\Big( Q_s^{(n)}h_s \int_0^s h_r \{Q^{(n)}_{r}\}^{-1}\ptr_r\, \vd B_r\Big) \, \vd s \mequal P_{T-t}f(X_t)\int_0^t\Big\langle h_s Q_s^{(n)}\int_0^s h_r \{Q^{(n)}_{r}\}^{-1}\ptr_r\, \vd B_r, \ptr_s\vd B_s \Big\rangle.
  \end{align*}
  Combining these equations with \eqref{L-formula2} and \eqref{L-eq2},  we obtain
    \begin{align*}
    &\int_0^t(LP_{T-s}f)(X_s)\tt h_sh_s\, \vd s-
      \bd P_{T-t}f\l(Q_t^{(n)}\tilde{h}_t\int_0^th_s(Q_s^{(n)})^{-1}\ptr_s \vd B_s\r)(X_t)\\
    &\quad +\int_0^t \Hess _{P_{T-s}f}(N_s,N_s)\Big\<N_s,\ Q_s^{(n)}\tt h_s\int_0^sh_r(Q_r^{(n)})^{-1}\ptr_r \vd B_r\Big\>\, \vd l_s\\
    &\quad+P_{T-t}f(X_t)\int_0^t\tt h_s\langle h_sZ_s, \ptr_s\vd B_s \rangle\\
    &\quad +P_{T-t}f(X_t)\int_0^t\Big\langle Q_s^{(n)}\int_0^s h_r (Q^{(n)}_{r})^{-1}\ptr_r\, \vd B_r, \ h_s \ptr_s\vd B_s \Big\rangle\mequal 0.
  \end{align*}
The last equation and     Eq.~\eqref{L-eq1} yield 
  \beq  \label{local-mar-2}\beg{split} 
    &(LP_{T-t}f)(X_t)\tt h_t^2- LP_Tf(x) -2\bd P_{T-t}f\l(Q_t^{(n)}\tt h_t\int_0^th_s(Q_s^{(n)})^{-1}\ptr_s \vd B_s\r)\\
    &\quad +2\int_0^t \Hess _{P_{T-s}f}(N_s,N_s)\Big\<N_s,\  Q_s^{(n)}\tt h_s\int_0^sh_r(Q_r^{(n)})^{-1}\ptr_r \vd B_r\Big\>\, \vd l_s\\
    &\quad+2P_{T-t}f(X_t)\int_0^t\tt h_s\langle h_sZ_s, \ptr_s\vd B_s \rangle\\
    &\quad +2P_{T-t}f(X_t)\int_0^t\Big\langle Q_s^{(n)}  \int_0^s h_r (Q^{(n)}_{r})^{-1}\ptr_r\, \vd B_r, \ h_s \ptr_s\vd B_s\Big\rangle\mequal 0.
  \end{split}\end{equation} 
   To get rid of the local martingales in the above calculations, we consider the diffusion process up to exit times from bounded balls.    Since   $\Hess _{ P_{\displaystyle\bf.}f}$ is locally bounded,  for any $k\ge 1$  we find a constant $c_k>0$ such that 
  \begin{align}\label{HessPf-NN}
    &\Bigg|\E\l[\int_0^{T\wedge \tau_k}\Hess _{P_{T-s}f}(N_s,N_s)\,\bigg\<N_s,\ Q_s^{(n)}\tt h_s\int_0^sh_r(Q_r^{(n)})^{-1}\ptr_r \vd B_r\bigg\> \,\d l_s\r]\Bigg| \notag \\
    &\leq c_k  \bigg( \E \int_0^{T\wedge \tau_k}|(Q_s^{(n)})^{*}N_s|^2\, \vd l_s  \bigg)^{1/2}
    \bigg(   \E \int_0^{T\wedge \tau_k}\l|\int_0^sh_r(Q_r^{(n)})^{-1}\ptr_r\vd B_r\r|^2\vd l_s \bigg)^{1/2}.
  \end{align}
 Next, we observe from \eqref{QNU}, \eqref{QNU'} and \eqref{TK} that
  \begin{align*}
    &\lim_{n\rightarrow \infty}
    \E\int_0^{T\wedge \tau_k}|(Q_s^{(n)})^{*}N|^2\, \vd l_s=  \lim_{n\rightarrow \infty}
    \E\int_0^{T\wedge \tau_k}\big|P_N Q_s^{(n)} (Q_s^{(n)})^{*}N\big|^2\, \vd l_s \\
    &\le \lim_{n\rightarrow \infty}\E\bigg[\sup_{t\in [0,T\land\tau_k]} |Q_t^{(n)}|^2 \int_0^{T\land\tau_k} \big|P_N Q_s^{(n)} \big|^2 \,\d s\bigg]=0,
  \end{align*}
  and  by Burkholder-Davis-Gundy's inequality,
  \beg{align*} \E\l[\int_0^{T\wedge \tau_k}\l|\int_0^sh_r(Q_r^{(n)})^{-1}\ptr_r\vd B_r\r|^2\vd l_s\r]\le \E\bigg[l_{T\land \tau_k} \sup_{s\in [0, T\land\tau_k]} \l|\int_0^sh_r(Q_r^{(n)})^{-1}\ptr_r\vd B_r\r|^2\bigg]<\infty.\end{align*}
 Using these estimates when taking the upper bound of \eqref{HessPf-NN},   we arrive at 
  \beq\label{NN}
    \lim_{n\rightarrow \infty}
    \E\l[\int_0^{T\wedge \tau_k}\Hess _{P_{T-s}f}(N_s,N_s)\,\bigg\<N_s,\ Q_s^{(n)}\tt h_s\int_0^sh_r(Q_r^{(n)})^{-1}\ptr_r \vd B_r\bigg\> \,\vd l_s\r]=0.
  \end{equation}

  On the other hand, since the process in \eqref{local-mar-2}
  is a  martingale  up to time  $ T\wedge\tau_k$, its expectation at  time 
  $T \wedge \tau_k$ vanishes, so that 
  \begin{align*}
    &L(P_Tf)(x)-\E\bigg[\tt h_{T\land\tau_k}^2 L P_{T-T\land\tau_k} f(X_{T\land\tau_k})-2\bd P_{T-T\land\tau_k}f\Big(Q_t^{(n)}\tt h_{T\land\tau_k}  \int_0^{T\land\tau_k} h_s(Q_s^{(n)})^{-1}\ptr_s \vd B_s\Big)\bigg]\\
    &=2\E\l[ f(X_{T})\int_0^{T\wedge \tau_k}\Big\langle  Q_{s}^{(n)}\int_0^{s}h_r(Q_{r}^{(n)})^{-1}\ptr_r\vd B_r, h_s\ptr_s\vd B_s \Big\rangle\r]+2\E \l[f(X_{T})\int_0^{T\wedge \tau_k}\langle \tt h_s h_sZ_s,\ptr_s\vd B_s \rangle\r]\\
    &\quad+2\E\l[\int_0^{T\wedge\tau_k}\Hess _{P_{T-s}f}(N_s,N_s)\, \Big\<N_s,\   \tt h_sQ^{(n)}_{s}\int_0^sh_r(Q^{(n)}_{r})^{-1} \ptr_r\,\vd B_r\Big\>\, \vd l_s\r].
  \end{align*}
By $\tt h_T=0$,  Lemma \ref{L1},  \eqref{NN}  and \eqref{CO}, we may first take $k\to\infty$ and then choose a subsequence $n_m\to\infty$ to derive 
    \eqref{BSL} for $f\in C_N^\infty(L)$.\smallskip

(2) To extend the formula to $f\in \mathcal B_b(M)$, we let $h_t=0$ for $t\ge T$ and  define  finite signed measures $\mu_{\vv}$  on $M$ as 
$$\mu_\vv(A):= 2\,\E\l[1_A(X_{T+\vv}^x)\l(M_T^{h}+ \int_0^{T}\bigg\langle \tt h_s  h_sZ(X_s^x),\,  \ptr_s\vd B_s\bigg\rangle\r)\r],\quad \vv\ge 0,$$
for measurable subsets $A\subset M$. By step (1) and \eqref{BC}, we have 
$$\ff{P_{T+\vv}f(x)- P_Tf(x)}\vv =\ff 1 \vv\int_0^\vv LP_{r+T}f(x) \,\d r =\ff 1 \vv\int_0^\vv\,\d r \int_M f\,\d\mu_{r}=\int_Mf\,\d\mu^{(\vv)},\ \ f\in \mathcal{C}_N^\infty(L),\ \vv>0,$$
where $\mu^{(\vv)}:=\ff 1 \vv\int_0^\vv \mu_r\,\d r$ is a finite signed measure on $M$. Since functions in $\mathcal{C}_N^\infty(L)$   determine finite measures, according to  Lemma \ref{L1} and condition \eqref{BC}, $\mu_{\vv}$ is a finite measure, and this implies
(we have in particular $M_{T+r}^{h}= M_T^{h}$ since  $h_t=0$ for $t\ge T$),
\beg{align*} & \ff{P_{T+\vv}f(x)- P_Tf(x)}\vv = \int_Mf\,\d\mu^{(\vv)}\\
&= \ff 1 \vv\int_0^\vv   2\E\l[f(X_{T+r})\bigg(M_{T}^{h}+\int_0^{T}\Big\langle \tt h_s h_sZ_s,\  \ptr_s\vd B_s\Big\rangle\bigg)\r]\,\d r,
\quad f\in \B_b(M),\ \vv>0.\end{align*}
Since the law of $X_T$ is absolutely continuous and $P_r f\to f$ a.e.~as $r\downarrow 0$,
we get by the strong Markov property $\E^{\SF_T}f(X_{T+r})=P_rf(X_T)\to f(X_T)$ a.s.~as
$r\downarrow 0$. 
 By the dominated convergence theorem we may let  $\vv\downarrow 0$ to  arrive at
\beq\label{BC'} \ff{\d P_t f(x)}{\d t}\Big|_{t=T}=  2\E\l[f(X_T)\bigg(M_T^{h}+ \int_0^{T}\Big\langle \tt h_s h_sZ_s,\  \ptr_s\vd B_s\Big\rangle\bigg)\r] ,\ \ f\in \B_b(M).\end{equation}  
On the other hand, for any $f\in \B_b(M)$ and $\vv>0$, $P_{t+\vv}f(x)$ is $C^1$ in $t\ge 0 $ and $C^2$ in $x$ with $NP_{t+\vv}f|_{\pp M}=0$. Hence, by It\^o's formula applied
to the process $(\phi P_T f)(X_t)$ for some cut-off function $\phi$ at $x$, the proof of (3.1.5) in \cite{Wbook14} implies 
\beq\label{BC''} LP_tf(x)=\ff{\d}{\d t } P_tf (x),\ \ t>0, \   f\in \B_b(M).\end{equation} 
 Combining \eqref{BC'} and  \eqref{BC''}, we prove \eqref{BSL} for all $f\in \B_b(M).$
\end{proof}

\begin{remark}
  When reduced to the case without boundary, our estimate still
  improves the result in \cite{Th19}. Moreover, compared to
  the estimate in \cite{Th19}, Theorem \ref{Hessian-th} only uses the lower bound of $\Ric_Z$
  instead of boundedness of $\Ric_Z$.
\end{remark}
Under curvature condition $\bf (A)$, with the particular choice $h_s:=-1/T$ for $s\in [0,T]$,
we obtain

\begin{corollary}\label{cor1}
  Assume that condition {\bf (A)} holds and $\|Z\|_{\infty}<\infty$.  Let $x\in M$ and $T>0$. Then
  \begin{align*}
    |L(P_Tf)|(x)
             &\leq  2\|f\|_{\infty}\l(\frac{\sqrt{3}\|Z\|_{\infty}}{3\sqrt{T}}+\frac{(3+\sqrt{10})\big(\E[\e^{\sigma^-l_{T}}]\big)^{1/2}\e^{K^-T/2}}{T}\r),
  \end{align*}
 for $f\in \mathcal{B}_b(M)$.  If $\sigma\geq 0$, then for $f\in \mathcal{B}_b(M)$,
   \begin{align}\label{convex-case-ineq}
    |L(P_Tf)|(x)
             &\leq (P_Tf^2)^{1/2}(x)\l(\frac{2\sqrt{3}\|Z\|_{\infty}}{3\sqrt{T}}+\frac{\sqrt{2}\e^{K^-T/2}}{T}\r).
  \end{align}
\end{corollary}

\begin{proof}
The first assertion is a direct consequence of inequality \eqref{BL0}.
It hence suffices to show inequality \eqref{convex-case-ineq}. Note that
\begin{align}\label{est-LPf}
&\left|\E^x\left[f(X_T^x)M_T^{(h,n)}\right]\right| 
\leq (P_Tf^2)^{1/2}(x) \l[\E\big | M_T^{(h,n)} \big |^2\r]^{1/2},
\end{align}
where $M^{(h,n)}$ is defined as in \eqref{MN}.
Let $h_s=-{1}/{T}$. Then,
\begin{align*}
\l[\E\big | M_{T\wedge \tau_k}^{(h,n)} \big |^2\r]^{1/2}&\leq \frac{1}{T^2} \l[  \int_0^{T}\E \bigg|  Q_{s\wedge \tau_k} ^{(n)} \int_0^{s\wedge \tau_k}  \{Q_{r}^{(n)}\}^{-1}\ptr_r \vd B_r\bigg|^2\,  \vd s \r]^{1/2}.
\end{align*}
By It\^{o}'s formula, we see that 
\begin{align}\label{term-L2-est}
\vd \l( \e^{-K^-s}\Big| Q_s^{(n)}\int_0^s(Q_{r}^{(n)})^{-1}\ptr_r\vd B_r\Big|^2\r)\leq 2\e^{-K^-s} \big\langle Q_s^{(n)}\int_0^s\{Q_{r}^{(n)}\}^{-1}\ptr_r\vd B_r,\  \ptr_s\vd B_s \big\rangle+\vd s.
\end{align}
For $0<s\leq \tau_k$, this implies that 
\begin{align*}
\E \l[ \e^{-K^-{s\wedge \tau_k}}\Big| Q^{(n)}_{s\wedge \tau_k}\int_0^{s\wedge \tau_k}\{Q_{r}^{(n)}\}^{-1}\ptr_r\vd B_r\Big|^2\r]\leq s.
\end{align*}
Letting $k$ tend to $\infty$ yields 
\begin{align*}
\E \l[ \Big| Q^{(n)}_{s\wedge \tau_k}\int_0^{s\wedge \tau_k}\{Q_{r}^{(n)}\}^{-1}\ptr_r\vd B_r\Big|^2\r]\leq s\e^{K^-{s}}.
\end{align*}
Combining this with \eqref{term-L2-est} and \eqref{est-LPf}, we see that  $\{M_T^{(h,n)}\}_{n\geq 1}$ is  bounded in $L^2(\P)$, and thus obtain a subsequence converging weakly to a random variable $M_T^{h}$ in $L^2(\P)$ and satisfying
\begin{align*}
\E\l[|M_T^{h}|^2\r]\leq \frac{\e^{K^-T}}{2T^2}.
\end{align*}
By this and the Bismut formula \eqref{BSL},  the second assertion \eqref{convex-case-ineq} holds.
\end{proof}

\section{Hessian formula for $P_tf$ and its application}\label{HessPf}

To state the main result of this Section, we first introduce some curvature conditions. 
For $x\in M$ and $v_1\in T_x M$, let $\Ric_Z^{\sharp}(v_1)\in T_xM$ be given by
$$  \langle \Ric_Z^{\sharp}(v_1), v_2  \rangle:= \Ric_Z(v_1,v_2)= \Ric(v_1,v_2)-\langle \nn_{v_1}Z, v_2\rangle,\quad v_2\in T_x M. $$
Let $R$ denote the Riemann curvature tensor.  Then $\bd^*R(v_1)$ is the linear operator on $T_xM$   determined by
\begin{align*}
  \langle \bd^*R(v_1,v_2),v_3\rangle  
  =\langle(\nabla_{v_3}\Ric^{\sharp})(v_1),v_2\rangle
  -\langle(\nabla_{v_2}\Ric^{\sharp})(v_3),v_1\rangle,\ \ v_3\in T_xM,
\end{align*}
where we write $\bd^*R(v_1,v_2)\equiv\bd^*R(v_1)v_2$.
Moreover, let    $R(v_1): T_xM\otimes T_xM\to T_xM$  be  given by
$$\langle R(v_1)(v_2,v_3), v_4\rangle :=\langle R(v_1,v_2)v_3, v_4\rangle,\ \ v_2, v_3, v_4\in T_xM.$$ Finally, let $|\cdot|$ be the operator  norm on  tensors, and $\|\cdot\|_\infty$ be the uniform norm of $|\cdot|$ over $M$.   

Assume that there exist two functions $K\in C(M)$ and
$\sigma \in C(\partial M)$ such that
\begin{align}\label{CD-1}\Ric_Z:=\Ric-\nabla Z\geq K,\quad  \  -\nabla N \geq \sigma,\end{align}
where the second condition means  $-\langle \nabla _X N, X \rangle \geq \sigma(x) |X|^2$ for $x\in \partial M$ and $X\in T_xM$.
Moreover,  assume that there exist three non-negative functions $\alpha, \, \beta$  and $\gamma$, such that
\begin{align}\label{CD-2}
|R|_{\HS} (x)\leq \alpha(x), \quad |{\bf d} ^*R+\nabla \Ric_Z^{\sharp}-R(Z)|(x)<\beta(x) ,\quad |\nabla (\nabla N)^{\sharp}+R(N)|(x)<\gamma(x),
\end{align}
 where for $x\in M$ and $v_1,v_2 \in T_xM$, 
\begin{align*}
|R|_{\HS}(x)=\sup\l\{|R^{\sharp, \sharp}(v_1,v_2)|_{\HS}(x): v_1,v_2\in T_xM,\, |v_1|\leq 1,\, |v_2|\leq  1\r\}.
\end{align*}

To establish the Hessian formula for  $P_tf$, we first introduce an operator $\tilde{Q}_t:T_xM\rightarrow T_{X_t}M$ defined by
\begin{align}\label{tilde-Q}
  {\rm D}\tilde{Q}_t=-\frac12\Ric_Z^{\sharp}(\tilde{Q}_t)\, \vd t+\frac12 (\nabla N)^ {\sharp}(\tilde{Q}_t)\, \vd l_t, \quad  \tilde{Q}_0=\id.
\end{align}
Then let the operator-valued process $W^{\tilde{h}}_t\colon T_xM\otimes T_xM\rightarrow T_{X_t(x)}M$ be defined as
solutions to the following covariant It\^{o} equation
\begin{align*}
  {\rm D}W^{\tilde{h}}_t(v, v)&=R(\ptr_t\vd B_t, \tilde{Q}_t(\tilde{h}(t)v))\tilde{Q}_t(v)\\
            &\quad-\frac12 (\bd^*R-R(Z)+\nabla \Ric_Z)^{\sharp}(\tilde{Q}_t(\tilde{h}(t)v), \tilde{Q}_t(v))\,\vd t\\
            &\quad-\frac12 (\nabla^2N-R(N))^{\sharp}(\tilde{Q}_t(\tilde{h}(t)v), \tilde{Q}_t(v))\,\vd l_t\\
            &\quad-\frac12 \Ric_Z^{\sharp}(W_t^{\tilde{h}}(v, v))\,\vd t+\frac12 (\nabla N)^{\sharp}(W_t^{\tilde{h}}(v, v))\,\vd l_t
\end{align*}
with initial condition $W_0^{\tilde{h}}(v,v)=0$.

\begin{theorem}\label{local-them3}
  Let $D$ be an open relatively compact subset of $M$,  $T>0$ and $x\in D$.
Suppose that \eqref{CD-1} and \eqref{CD-2} hold. Let $h(\cdot)$  
be an adapted and bounded real process such that $\int_0^t h_s\, \vd s=-1$ for $t\geq T\wedge \tau_D$, and such that
\begin{align}\label{CD-h}
&\E \Bigg[ \int_0^{T\wedge \tau_D} \left(h^2_s+\tilde{h}^2_s  (\alpha^2 (X_s)+\beta^2(X_s))\right)\e^{\int_0^sK^-(X_r)\, \vd r+ \int_0^s \sigma^-(X_r)\, \vd  l_r} \, \vd s \notag \\
&\qquad \qquad\qquad    +\int_0^{T\wedge \tau_D}\tilde{h}^2_s \gamma^2(X_s)\e^{\int_0^sK^-(X_r)\, \vd r+ \int_0^s \sigma^-(X_r)\, \vd  l_r} \, \vd l_s \Bigg] <\infty
\end{align}
where $\tilde{h}(t)=1+\int_0^th_s\, \vd s$.

 Then for
  $f\in \mathcal{B}_b(M)$ and $v\in T_xM$,
  \begin{align*}
    \Hess_{ P_{T}f} (v, v)(x)=
    &- \E^x \l[f(X_T)\int_0^T\langle W_s^{\tilde{h}}(v, h_s v),\ptr_s\vd B_s \rangle \r] \notag\\
    &+ \E^x \l[ f(X_T)\l(\l(\int_0^T\langle \tilde{Q}_s(h_sv),
      \ptr_s \vd  B_s \rangle\r)^2-\int_0^T |\tilde{Q}_s(h_s v)|^2  \,\vd s\r) \r].
  \end{align*}
Moreover, 
\begin{align}\label{Hess_P_Tf_Estimate}
| \Hess_{ P_{T}f } | & \leq 3 \|f\|_{\infty} C(h)^{1/2} \Bigg\{ (3+\sqrt{10})\l[\E\l(\int_0^{T\wedge \tau_D} \alpha^2(X_s)\e^{\int_0^sK^-(X_r)\, \vd r+ \int_0^s \sigma^-(X_r)\, \vd  l_r}\tilde{h}_s^2\, \vd s\r)\r]^{1/2}\notag\\
&\qquad + \frac{1}{2}\l[\E\l(\int_0^{T\wedge \tau_D}\beta^2(X_s) \e^{\int_0^sK^-(X_r)\, \vd r+ \int_0^s \sigma^-(X_r)\, \vd  l_r} \tilde{h}_s^2\, \vd s\r)\r]^{1/2} \notag\\
&\qquad + \frac{1}{2}\l[\E\l(\int_0^{T\wedge \tau_D}\gamma^2(X_s) \e^{\int_0^sK^-(X_r)\, \vd r+ \int_0^s \sigma^-(X_r)\, \vd  l_r}\tilde{h}_s^2 \, \vd l_s\r)\r]^{1/2}+\frac{2}{3}C(h)^{1/2}\Bigg\},
\end{align}
where $C(h)$ is defined as in \eqref{def-Ch}.  
\end{theorem} 

\begin{remarks}\label{rem1}
1) \ The original idea of the proof for the Hessian formula comes from Elworthy-Li~\cite{EL94} and Thalmaier \cite{APT2003}. 
Our form of the formula is consistent  with \cite{CCT} with the choice of one random test function $h$ only.  The main difficulty here is to deal with
the impact of the boundary  and to weaken the conditions on the curvature and the process $h$. Theorem \ref{local-them3} also improves the results in \cite{APT2003, CCT} and gives a new estimates even when the boundary is empty. 

2) \ Let $D$ be an open relatively compact subset of $M$. Assume that $h$ is an adapted and non-positive process with $h_s=0$ for $s\geq T\wedge \tau_D$ and $\int_0^T h_s\, \vd s =-1$, which imply $\tilde{h}_s=0$ for $s\geq T\wedge \tau_D$.  Then the functions $K$, $\sigma$, $\alpha$, $\beta$ and $\gamma$ are all bounded on $D$ and $|\tt h _s |\leq 1$.  Moreover, condition \eqref{CD-h} can be simplified to 
\begin{align*}
\E\l[\int_0^T h_s^{2}\e^{ \int_0^s \sigma^-(X_r)\, \vd  l_r}\, \vd s\r]<\infty.
\end{align*}
The corresponding result is then the local version of the Hessian formula for the heat semigroup.

3) \ As $-\langle \nabla_N N, N \rangle=0$,
we know that  $-\nabla N \geq \sigma$ implies  $\sigma \leq 0$.

4) \ Assume $\II
\geq \sigma_1$ and $|\nabla_N N|\leq \sigma_2$. Then  for $X\in T_xM$ and $x\in \partial M$, 
$X=X_1+X_2$ such that  $X_1\in T_x \partial M$ and $X_2=\langle X, N \rangle N$,
\begin{align*}
-\langle \nabla _XN, X \rangle &=- \langle \nabla _{X_1} N, X_1  \rangle - \langle X, N\rangle^2 \langle \nabla _N N, N \rangle -  \langle X, N\rangle \langle \nabla _{X_1} N,  N \rangle -\langle X, N\rangle \langle \nabla _{N} N,  X_1 \rangle \\
&=-\langle \nabla _{X_1} N, X_1  \rangle - \langle X, N\rangle \langle \nabla _{N} N,  X_1 \rangle\\
&\geq \sigma_1 |X_1|^2 -\sigma_2 |X_1|\cdot |\langle X, N\rangle|\geq \min\l\{\sigma_1,\, -\frac{\sigma_2}{2}\r\} |X|^2.
\end{align*}
In particular, if $\nabla_N N=0$ and $\II\geq  \sigma$, then it is easy to see that for $x\in \partial M$,
$$-\langle\nabla_X N, X\rangle(x)\geq -\sigma^- |X|^2,\ \ \ \ \mbox{for}\ \ X\in T_xM \  \mbox{and} \ x\in \partial M.$$

5) \ Naturally, one might try to work with $Q_t$ instead of $\tilde{Q}_t$ to define $M_t(v,v)$, in order to avoid the term $\nabla N$. But we have already seen that
$Q_t$ is the limit of $Q^{(n)}_t$, see the proof of \cite[Theorem 3.2.1]{Wbook14}, which satisfy the covariant It\^{o} equation \eqref{QN}.
We have
\begin{align}\label{reason-1}
& \vd \l( \nabla \bd P_{T-t}f(Q_t^{(n)}(v), Q_t^{(n)}(v))(X_t)\r)\notag\\
&= \left(\nabla_{\ptr_t\,\vd B_t}\nabla \bd P_{T-t}f\right)\left(Q_t^{(n)}(v)
,Q_t^{(n)}(v)\right)- (\nabla \bd P_{T-t}f)\l(\Ric_Z^{\sharp} (Q_t^{(n)}(v)), Q_t^{(n)}(v)\r)\,\vd t\notag\\
&\quad
- (\nabla \bd P_{T-t}f)\l(Q_t^{(n)}(v), \II^{\sharp} (Q_t^{(n)}(v))\r)\,\vd l_t+\partial_t(\nabla \bd P_{T-t}f)(Q_t^{(n)}(v), Q_t^{(n)}(v))\, \vd t\notag\\
&\quad+\frac12 \nabla_N(\nabla \bd P_{T-t}f)
(Q_t^{(n)}(v), Q_t^{(n)}(v))\,\vd l_t
+\frac12 (\tr\nabla^2+\nabla_Z)(\nabla \bd P_{T-t}f)
(Q_t^{(n)}(v), Q_t^{(n)}(v))\,\vd t\notag\\
&\quad-n\langle Q_t^{(n)}(v)), N(X_t) \rangle \nabla \bd P_{T-t}f( N(X_t), Q_t^{(n)}(v))\,\vd l_t\notag\\
&\mequal -n\langle Q_t^{(n)}(v)), N(X_t) \rangle \, \nabla \bd P_{T-t}f(Q_t^{(n)}(v),N)\,\vd l_t\notag\\
&\quad -\langle Q_t^{(n)}(v)), N(X_t) \rangle \, \nabla \bd P_{T-t}f( \nabla _N N, Q_t^{(n)}(v) )\,\vd l_t\notag\\
&\quad+\frac12 (\bd P_{T-t}f)( (\nabla (\nabla N)^{\sharp}+R(N))(Q^{(n)}_t(v),Q^{(n)}_t(v)))\, \vd l_t\notag\\
&\quad+\frac12 (\bd P_{T-t}f) \big((\bd^*R-R(Z)+\nabla \Ric_Z^{\sharp})(Q^{(n)}_t(v),Q^{(n)}_t(v))\big)\,\vd t \notag \\
&\quad -\nabla\bd P_{T-t}f(R^{\sharp, \sharp}(Q^{(n)}_t(v), Q^{(n)}_t(v)))\, \vd t.
\end{align}
The main difficulty is to clarify the limit of
$$n\langle Q_t^{(n)}(v)), N(X_t) \rangle \, \nabla \bd P_{T-t}f(Q_t^{(n)}(v),N), $$
as $n$ tends to $\infty$.
In addition, information on $\nabla_N N$ is required to deal with the term $$\langle Q_t^{(n)}(v)), N(X_t) \rangle \, \nabla \bd P_{T-t}f( \nabla _N N, Q_t^{(n)}(v) )\,\vd l_t.$$
To this end, we even need information concerning $\nabla N$ on the full vector bundle of the boundary
if we use $Q_t$ in the definition of $M_t(v,v)$ in the above proof, and then it is still non-trivial to check the martingale property. In this respect, working with the functional $\tilde{Q}_t$ instead of $Q_t$  not only  simplifies the calculation, it also doesn't require additional conditions.
\end{remarks}

To prove Theorem \ref{local-them3}, we need the following two lemmata.

\begin{lemma}\label{lem4}
Keeping the assumptions as in Theorem \ref{local-them3},  we have
\begin{align*}
& \E\l[\sup_{t\in [0,T]} \Big | \int_0^t h_s \langle W_s^{\tilde{h}}(v,v), \, \ptr_s \vd B_s  \rangle \Big| \r]  \\
&\leq  3C(h)^{1/2} \Bigg\{ (3+\sqrt{10})\l(\E\int_0^T \alpha^2(X_s)\e^{\int_0^sK^-(X_r)\, \vd r+ \int_0^s \sigma^-(X_r)\, \vd  l_r}\tilde{h}_s^2\, \vd s\r)^{1/2}\\
&\qquad + \frac{1}{2}\l(\E\int_0^T\beta^2(X_s) \e^{\int_0^sK^-(X_r)\, \vd r+ \int_0^s \sigma^-(X_r)\, \vd  r}\tilde{h}_s^2 \, \vd s\r)^{1/2} \\
&\qquad + \frac{1}{2}\l(\E\int_0^T\gamma^2(X_s) \e^{\int_0^sK^-(X_r)\, \vd r+ \int_0^s \sigma^-(X_r)\, \vd  l_r}\tilde{h}_s^2\, \vd l_s\r)^{1/2}\Bigg\}.
\end{align*}
\end{lemma}

\begin{proof}
By the Lenglart inequality and the Minkowski inequality, we have
\begin{align}\label{W-est}
& \E\l[\sup_{t\in [0,T]} \Big | \int_0^{t\wedge \tau_k} h_s \langle W_s^{\tilde{h}}(v,v), \, \ptr_s \vd B_s  \rangle \Big| \r] \\
&\leq  3\E  \l[ \int_0^{T\wedge \tau_k} h_s^2|W_s^{\tilde{h} }(v,v)|^2 \, \vd s  \r]^{1/2} \notag\\
&\leq 3 \E\l[\l(\int_0^{T\wedge \tau_k}h_s^2 \Big|\tilde{Q}_{s}\int_0^{s}\tilde{Q}_r^{-1} R(\ptr_r \vd B_r, \tilde{Q}_r(\tilde{h}(r)v))\tilde{Q}_r(v)\Big|^2\, \vd s \r)^{1/2}\r] \notag \\
&\quad+\frac{3}{2}\E\l[\l(\int_0^{T\wedge \tau_k} h_s^2  \Big|\tilde{Q}_{s}\int_0^{s} \tilde{Q}_r^{-1} (\bd^*R-R(Z)+\nabla \Ric_Z)^{\sharp}(\tilde{Q}_r(\tilde{h}(r)v), \tilde{Q}_r(v))\,\vd r \Big|^2\, \vd s \r)^{1/2}\r] \notag\\
&\quad+ \frac {3}2  \E\l[ \l(\int_0^{T\wedge \tau_k} h_s^2 \Big|\tilde{Q}_{s}\int_0^{s} \tilde{Q}_r^{-1} (\nabla^2N-R(N))^{\sharp}(\tilde{Q}_r(\tilde{h}(r)v), \tilde{Q}_r(v))\,\vd l_r \Big|^2\, \vd s\r)^{1/2}\r] .
\end{align}
Let 
\begin{align*}
& \xi_s^{(1)}= \tilde{Q}_{s}\int_0^{s}\tilde{Q}_r^{-1} R(\ptr_r \vd B_r, \tilde{Q}_r(\tilde{h}(r)v))\tilde{Q}_r(v); \\
& \xi_s^{(2)}= \tilde{Q}_{s}\int_0^{s} \tilde{Q}_r^{-1} (\bd^*R-R(Z)+\nabla \Ric_Z)^{\sharp}(\tilde{Q}_r(\tilde{h}(r)v), \tilde{Q}_r(v))\,\vd r; \\
& \xi_s^{(3)}= \tilde{Q}_{s}\int_0^{s} \tilde{Q}_r^{-1} (\nabla^2N-R(N))^{\sharp}(\tilde{Q}_r(\tilde{h}(r)v), \tilde{Q}_r(v))\,\vd l_r.
\end{align*}
Then,  we have
\begin{align}\label{term-1}
& \E\l[\l(\int_0^{T\wedge \tau_k}h_s^2 |\xi_s^{(1)}|^2\, \vd s \r)^{1/2}\r] \notag \\
&\leq  \E \l[\sup_{s\in [0,{T\wedge \tau_k}]}  |\xi_s^{(1)}|^2 \e^{-\int_0^sK^-(X_r)\, \vd r- \int_0^s \sigma^-(X_r)\, \vd  l_r}\r]^{1/2}
\E\l[\int_0^{T\wedge \tau_k} h_s^2\e^{\int_0^sK^-(X_r)\, \vd r+ \int_0^s \sigma^-(X_r)\, \vd  l_r}\, \vd s\r]^{1/2},
\end{align}
and
\begin{align}\label{xi-1}
 &\vd \Big|\e^{-\frac{1}{2}\int_0^sK^-(X_r)\, \vd r-\frac{1}{2}\int_0^s \sigma^-(X_r)\, \vd  l_r} \xi_s^{(1)}\Big|^2 \notag\\
  &=2\e^{-\int_0^sK^-(X_r)\, \vd r- \int_0^s \sigma^-(X_r)\, \vd  l_r}\Big\langle R(\ptr_s\vd B_s, \tilde{Q}_s(\tilde{h}_sv))\tilde{Q}_s(v),  \xi^{(1)}_s\Big\rangle \notag\\
&\quad+\e^{-\int_0^sK^-(X_r)\, \vd r- \int_0^s \sigma^-(X_r)\, \vd  l_r}\big|R^{\sharp,\sharp}(\tilde{Q}_s(\tilde{h}_sv),\tilde{Q}_s(v))\big |_\HS^2\,\vd s \notag \\
&\quad-\e^{-\int_0^sK^-(X_r)\, \vd r- \int_0^s \sigma^-(X_r)\, \vd  l_r}\Ric_Z\big(\xi^{(1)}_s, \xi^{(1)}_s\big)\,\vd s -\e^{-\int_0^sK^-(X_r)\, \vd r -\int_0^s \sigma^-(X_r)\, \vd  l_r}\l \langle-\nabla_{\xi^{(1)}_s} N, \xi_s^{(1)} \r\rangle\, \vd l_s \notag \\
&\quad-K^-(X_s) \Big|\e^{-\frac{1}{2}\int_0^sK^-(X_r)\, \vd r- \frac{1}{2}\int_0^s \sigma^-(X_r)\, \vd  l_r} \xi_s^{(1)}\Big|^2\,\vd s-\sigma^- (X_s) \Big|\e^{-\frac{1}{2}\int_0^sK^-(X_r)\, \vd r- \frac{1}{2}\int_0^s \sigma^-(X_r)\, \vd  l_r}\xi_s^{(1)}\Big|^2\,\vd l_s\notag \\
& \leq  2\e^{-\int_0^sK^-(X_r)\, \vd r- \int_0^s \sigma^-(X_r)\, \vd  l_r}\Big\langle R(\ptr_s\vd B_s, \tilde{Q}_s(\tilde{h}_sv))\tilde{Q}_s(v),  \xi^{(1)}_s\Big\rangle \notag \\
&\quad +\e^{-\int_0^sK^-(X_r)\, \vd r- \int_0^s \sigma^-(X_r)\, \vd  l_r}\big|R^{\sharp,\sharp}(\tilde{Q}_s(\tilde{h}_sv),\tilde{Q}_s(v))\big|_\HS^2\,\vd s \notag \\
& \leq 
2\e^{-\int_0^sK^-(X_r)\, \vd r- \int_0^s \sigma^-(X_r)\, \vd  l_r}\Big\langle R(\ptr_s\vd B_s, \tilde{Q}_s(\tilde{h}_sv))\tilde{Q}_s(v),  \xi^{(1)}_s\Big\rangle \notag \\
&\quad +\alpha(X_s)^2\e^{\int_0^sK^-(X_r)\, \vd r+ \int_0^s \sigma^-(X_r)\, \vd  l_r} \tilde{h}_s^2\, \vd s,\quad s<\tau_k,
\end{align}
which implies
\begin{align*}
 & \E \l[\sup_{s\in [0,{T\wedge \tau_k}]}  |\xi_s^{(1)}|^2 \e^{-\int_0^sK^-(X_r)\, \vd r- \int_0^s \sigma^-(X_r)\, \vd  l_r}\r] \\
&\leq 6 \E \l[ \l(\int_0^{T\wedge \tau_k} |\xi_s^{(1)}|^2 \alpha(X_s)^2 \tilde{h}_s^2 \, \vd s   \r)^{1/2} \r]+ \E\l[\int _0^{T\wedge \tau_k} \alpha(X_s)^2 \e^{\int_0^sK^-(X_r)\, \vd r+ \int_0^s \sigma^-(X_r)\, \vd  l_r}\tilde{h}_s^2\, \vd s\r] \\
&\leq 6 \E \l[\l(\sup_{s\in [0,{T\wedge \tau_k}]}  |\xi_s^{(1)}|^2 \e^{-\int_0^sK^-(X_r)\, \vd r- \int_0^s \sigma^-(X_r)\, \vd  l_r}\r)^{1/2} \l(\int_0^{T\wedge \tau_k} \alpha(X_s)^2 \e^{\int_0^sK^-(X_r)\, \vd r+ \int_0^s \sigma^-(X_r)\, \vd  l_r}\tilde{h}_s^2 \, \vd s  \r)^{1/2}\r]\\
&\quad +\E \l[ \int_0^{T\wedge \tau_k} \alpha(X_s)^2 \e^{\int_0^sK^-(X_r)\, \vd r+ \int_0^s \sigma^-(X_r)\, \vd  l_r}\tilde{h}_s^2 \, \vd s   \r]\\
&\leq  \frac{1}{2\delta} \E \l[\sup_{s\in [0,{T\wedge \tau_k}]}  |\xi_s^{(1)}|^2 \e^{-\int_0^sK^-(X_r)\, \vd r- \int_0^s \sigma^-(X_r)\, \vd  l_r}\r] + 18\delta \E \l[\int_0^{T\wedge \tau_k} \alpha(X_s)^2 \e^{\int_0^sK^-(X_r)\, \vd r+ \int_0^s \sigma^-(X_r)\, \vd  l_r}\tilde{h}_s^2 \, \vd s  \r]\\
&\quad +\E \l[ \int_0^{T\wedge \tau_k} \alpha(X_s)^2 \e^{\int_0^sK^-(X_r)\, \vd r+ \int_0^s \sigma^-(X_r)\, \vd  l_r}\tilde{h}_s^2 \, \vd s   \r], \quad \delta>0.
\end{align*}
Substituting the optimal choice $\delta=\ff 1 6\big(3+\ss{10}\big)$,  we get
\begin{align*}
& \E \l[\sup_{s\in [0,T\wedge \tau_k]}  |\xi_s^{(1)}|^2 \e^{-\int_0^sK^-(X_r)\, \vd r- \int_0^s \sigma^-(X_r)\, \vd  l_r}\r] \\
&\leq  (3+\sqrt{10})^2\E \l[ \int_0^{T\wedge \tau_k}\alpha(X_s)^2 \e^{\int_0^sK^-(X_r)\, \vd r+ \int_0^s \sigma^-(X_r)\, \vd  l_r}\tilde{h}_s^2 \, \vd s   \r].
\end{align*}
Combining this with \eqref{term-1} and letting $k$ tend to $\infty$  yields
\begin{align*}
& \E\l(\l[\int_0^T h_s^2 |\xi_s^{(1)}|^2\, \vd s \r]^{1/2}\r) \\
&\leq  \E \l[\sup_{s\in [0,T]}  |\xi_s^{(1)}|^2 \e^{-\int_0^sK^-(X_r)\, \vd r- \int_0^s \sigma^-(X_r)\, \vd  l_r}\r]^{1/2}
\E\l[\int_0^{T} h_s^2\e^{\int_0^sK^-(X_r)\, \vd r+ \int_0^s \sigma^-(X_r)\, \vd  l_r}\, \vd s\r]^{1/2} \\
&\leq \Big(3+\sqrt{10}\Big) \E\l[\int_0^T\alpha(X_s)^2\e^{\int_0^sK^-(X_r)\, \vd r+ \int_0^s \sigma^-(X_r)\, \vd  l_r}\tilde{h}_s^2\, \vd s\r]^{1/2}C(h)^{1/2}.
\end{align*}
Moreover, for any $\eps>0$, 
\begin{align*}
&\vd \l(\Big|\e^{-\frac{1}{2}\int_0^tK^-(X_s)\, \vd t-\frac{1}{2} \int_0^t \sigma^-(X_s)\, \vd  l_s}\xi_t^{(2)} \Big|^2+\eps\r)^{1/2} = \frac{\vd \Big|\e^{-\frac{1}{2}\int_0^tK^-(X_s)\, \vd t-\frac{1}{2} \int_0^t \sigma^-(X_s)\, \vd  l_s}\xi_t^{(2)} \Big|^2}{2\l(\Big|\e^{-\frac{1}{2}\int_0^tK^-(X_s)\, \vd t-\frac{1}{2} \int_0^t \sigma^-(X_s)\, \vd  l_s}\xi_t^{(2)}  \Big|^2+\eps\r)^{1/2}}\\
& = \e^{-\int_0^tK^-(X_s)\, \vd t-\int_0^t \sigma^-(X_s)\, \vd  l_s}  \l(\Big|\e^{-\frac{1}{2}\int_0^tK^-(X_s)\, \vd t-\frac{1}{2} \int_0^t \sigma^-(X_s)\, \vd  l_s}\xi_t^{(2)}\Big|^2+\eps\r)^{-1/2} \\
&\quad  \times\Bigg[- (\Ric_Z+K^-(X_t)g)\Big(\xi_t^{(2)}, \xi_t^{(2)}\Big)\,\vd t + (\nabla N-\sigma^-(X_t)g)\Big(\xi_t^{(2)}, \xi_t^{(2)}\Big)\,\vd l_t \\
&\qquad  +   \Big \langle  (\bd^*R-R(Z)+\nabla \Ric_Z)^{\sharp}(\tilde{Q}_t(\tilde{h}(t)v), \tilde{Q}_t(v)), \xi_t^{(2)} \Big\rangle \,\vd t \Bigg] \\
&\leq  \e^{-\int_0^tK^-(X_s)\, \vd t- \int_0^t \sigma^-(X_s)\, \vd  l_s} \l(\Big|\e^{-\frac{1}{2}\int_0^tK^-(X_s)\, \vd s-\frac{1}{2} \int_0^t \sigma^-(X_s)\, \vd  l_s} \xi_t^{(2)} \Big|^2+\eps\r)^{-1/2}  \\
&\qquad \times  \Big \langle (\bd^*R-R(Z)+\nabla \Ric_Z)^{\sharp}(\tilde{Q}_t(\tilde{h}(t)v), \tilde{Q}_t(v)), \xi_t^{(2)} \Big \rangle \,\vd t  \\
&\leq \beta(X_t) \e^{\frac{1}{2}\int_0^tK^-(X_s)\, \vd t+ \frac{1}{2} \int_0^t \sigma^-(X_s)\, \vd  l_s}\tilde{h}(t)\,\vd t, \quad t<\tau_k.
\end{align*}
Taking the integral on both sides, letting $\eps$ tend to 0 and $k$ tend  to $\infty$, we then conclude that
\begin{align}\label{xi-2}
 &\Big|\xi_t^{(2)} \Big|\leq 
\e^{\frac{1}{2} \int_0^tK^-(X_s)\, \vd s+\frac{1}{2} \int_0^t \sigma^-(X_s)\, \vd  l_s} \int_0^t\beta(X_s) \e^{\frac{1}{2} \int_0^sK^-(X_r)\, \vd r+\frac{1}{2}  \int_0^s \sigma^-(X_r)\, \vd  l_r} \tilde{h}_s\, \vd s.
\end{align}
With a similar argument, we have
\begin{align}\label{xi-3}
 &\Big|\xi_t^{(3)} \Big|\leq 
\e^{\frac{1}{2}\int_0^tK^-(X_s)\, \vd t+\frac{1}{2}\int_0^t \sigma^-(X_s)\, \vd  l_s} \int_0^t\gamma(X_s) \e^{\int_0^sK^-(X_r)\, \vd r+ \int_0^s \sigma^-(X_r)\, \vd  l_r} \tilde{h}_s\, \vd l_s.
\end{align}
These estimates together imply 
\begin{align*}
& \E\l[\sup_{t\in [0,T]} \Big | \int_0^t h_s \langle W_s^{\tilde{h}}(v,v), \, \ptr_s \vd B_s  \rangle \Big| \r] \\
&\leq  3C(h)^{1/2} \Bigg\{ \big(3+\sqrt{10}\big)\l(\E\int_0^T\alpha(X_s)^2\e^{\int_0^sK^-(X_r)\, \vd r+ \int_0^s \sigma^-(X_r)\, \vd  l_r} \tilde{h}_s^2\, \vd s\r)^{1/2} \\
&\quad  \quad + \frac{1}{2}\l(\E\int_0^T\beta(X_s)^2 \e^{\int_0^sK^-(X_r)\, \vd r+ \int_0^s \sigma^-(X_r)\, \vd  r}\tilde{h}_s^2 \, \vd s\r)^{1/2}\\
&\quad  \quad + \frac{1}{2}\l(\E\int_0^T\gamma(X_s)^2 \e^{\int_0^sK^-(X_r)\, \vd r+ \int_0^s \sigma^-(X_r)\, \vd  l_r}\tilde{h}_s^2\, \vd l_s\r)^{1/2}\Bigg\}.\qedhere
\end{align*}
\end{proof}

\begin{lemma}\label{lem3}
  Let  $D$ be an open relatively compact domain in $M$ and $x\in D$.
  Fix $T>0$ and
  suppose that $h$ is a bounded, non-negative and adapted process with
  paths in the Cameron-Martin space $L^{1,2}([0,T];\R)$.  Then for
  $f\in \mathcal{B}_b(M)$ and $v\in T_xM$,
  \begin{align}\label{Loc_mart_Hessian}
    &(\nabla {\bf d} P_{T-t}f)(\tilde{Q}_t(\tilde{h}(t)v),\tilde{Q}_t(\tilde{h}(t)v))+({\bf d} P_{T-t}f)
      \big(W_t^{\tilde{h}}(v,\tilde{h}(t) v)\big) \notag\\
    &\quad -2{\bf d } P_{T-t}f(\tilde{Q}_t(\tilde{h}(t) v))\int_0^t\langle \tilde{Q}_s(h_sv),
      \ptr_s \vd  B_s \rangle
      \notag\\
    &\quad -P_{T-t}f(X_t)\int_0^t\langle W_s^{\tilde{h}}(v, h_s v),\ptr_s\vd B_s \rangle \notag\\
    &\quad +P_{T-t}f(X_t)\l(\l(\int_0^t\langle \tilde{Q}_s(h_sv),
      \ptr_s \vd  B_s \rangle\r)^2-\int_0^t |\tilde{Q}_s(h_sv)|^2  \,\vd s\r)
  \end{align}
  is a local martingale, and in particular a true martingale on $[0, T\wedge\tau_D)$.
\end{lemma}

\begin{proof}
 We first prove that  for
  $f\in \mathcal{B}_b(M)$ and $v\in T_xM$,
  \begin{align*}
    M_t(v,v)=\nabla {\bf d} P_{T-t}f(\tilde{Q}_t(v), \tilde{Q}_t(v))+({\bf d} P_{T-t}f)(W_t(v,v))
  \end{align*}
  is a local martingale where 
  \begin{align}\label{W-def}
W_t(v,v)=&\tilde{Q}_t\int_0^t \tilde{Q}_s^{-1} R(\ptr_s \vd B_s, \tilde{Q}_s(v))\tilde{Q}_s(v) \notag\\
&-\frac{1}{2} \tilde{Q}_t\int_0^t \tilde{Q}_s^{-1} (\bd^*R-R(Z)+\nabla \Ric_Z^{\sharp})(\tilde{Q}_s(v), \tilde{Q}_s(v))\,\vd s  \notag \\
&- \frac 12 \tilde{Q}_t\int_0^t \tilde{Q}_s^{-1} (\nabla(\nabla N)^{\sharp}-R(N))(\tilde{Q}_s(v), \tilde{Q}_s(v))\,\vd l_s. 
\end{align}
Let us recall some commutation rules
which will be helpful in the subsequent calculations:
\begin{align}
&\bd L f=(\tr \nabla^2 +\nabla _{Z} ) \bd f-\bd f(\Ric^{\sharp}-(\nabla Z)^{\sharp}); \label{commutation-1}\\
&\nabla \bd (\Delta f)=\tr \nabla^2(\nabla \bd f)-(\nabla \bd f)(\Ric^{\sharp}\odot\id+\id\odot\Ric^{\sharp}-2R^{\sharp, \sharp})-\bd f(\bd ^*R+\nabla \Ric^{\sharp});\label{commutation-3}\\
&\nabla \bd (Z(f))=\nabla_{Z}(\nabla \bd f)+(\nabla \bd f)((\nabla Z)^{\sharp}\odot\id+\id\odot(\nabla Z)^{\sharp})+\bd f(\nabla (\nabla Z)^{\sharp}+R(Z) );\label{commutation-4}\\
&\nabla \bd (Nf)=\nabla_N(\nabla \bd f)+(\nabla \bd f)((\nabla N)^{\sharp}\odot\id+\id\odot (\nabla N )^{\sharp})+\bd f( \nabla(\nabla N)^{\sharp}+R(N))\label{commutation-5}
\end{align}
where
$
\nabla \bd f(\nabla N\odot \id(v,v))=\nabla \bd f(\nabla_{v} N, v).
$
Let
\begin{align}\label{M-martingale}
M_t(v,v)=\nabla \bd P_{T-t}f(\tilde{Q}_t(v), \tilde{Q}_t(v))+(\bd P_{T-t}f)(W_t(v,v)).
\end{align}
Then by It\^{o}'s formula we have
\begin{align}\label{Ito-Nt}
\vd M_t(v,v)
&=(\nabla_{\ptr_t\,\bd B_t}\nabla \bd P_{T-t}f)(\tilde{Q}_t(v)
,\tilde{Q}_t(v))+(\nabla_{\ptr_t\vd B_t}\bd P_{T-t}f)(W_t(v,v))\notag\\
&\quad+(\nabla \bd P_{T-t}f)\l({\rm D}\tilde{Q}_t(v), \tilde{Q}_t(v)\r)+(\nabla \bd P_{T-t}f)\l(\tilde{Q}_t(v), {\rm D}\tilde{Q}_t(v)\r)\notag\\
&\quad+\partial_t(\nabla \bd P_{T-t}f)(\tilde{Q}_t(v), \tilde{Q}_t(v))\, \vd t+\frac12 (\tr\nabla^2+\nabla_Z)(\nabla \bd P_{T-t}f)
(\tilde{Q}_t(v), \tilde{Q}_t(v))\,\vd t\notag\\
&\quad+\frac12 \nabla_{N}(\nabla \bd P_{T-t}f)
(\tilde{Q}_t(v), \tilde{Q}_t(v))\,\vd l_t+\frac12 \nabla _{N}(\bd P_{T-t}f)(W_t(v,v))\,\vd l_t\notag\\
&
\quad+(\bd P_{T-t}f)({\rm D}W_t(v,v))+\langle {\rm D} (\bd P_{T-t}f), {\rm D}W_t(v,v)\rangle\notag\\
&\quad+\partial_t(\bd P_{T-t}f)(W_t(v,v))\, \vd t+\frac12 
(\tr \nabla^2+\nabla_Z)(\bd P_{T-t}f)(W_t(v,v))\,\vd t.
\end{align}
Taking into account the commutation properties \eqref{commutation-3}--\eqref{commutation-5} and according to the definition of $\tilde{Q}_t$, for the terms on the right side of \eqref{Ito-Nt}, we observe that
\begin{align*}
&\partial_t(\nabla \bd P_{T-t}f)(\tilde{Q}_t(v), \tilde{Q}_t(v))\, \vd t+(\nabla \bd P_{T-t}f)\l({\rm D}\tilde{Q}_t(v), \tilde{Q}_t(v)\r)+(\nabla \bd P_{T-t}f)\l(\tilde{Q}_t(v), {\rm D}\tilde{Q}_t(v)\r)\\
&=-\frac12 \l((\tr \nabla^2+\nabla_Z)\nabla \bd P_{T-t}f\r)(\tilde{Q}_t(v), \tilde{Q}_t(v))\, \vd t
\\
&\quad+ (\nabla \bd P_{T-t}f)(\Ric_Z^{\sharp}(\tilde{Q}_t(v)) ,\tilde{Q}_t(v))\, \vd t \\
&\quad-\nabla\bd P_{T-t}f(R^{\sharp, \sharp}(\tilde{Q}_t(v),\tilde{Q}_t(v)))\, \vd t+\frac12 \bd P_{T-t}f\big((\bd ^*R-R(Z)+\nabla \Ric_Z)^{\sharp}(\tilde{Q}_t(v),\tilde{Q}_t(v))\big)\, \vd t\\
&\quad+ (\nabla \bd P_{T-t}f)\l(-\Ric_Z^{\sharp}(\tilde{Q}_t(v))\, \vd t+(\nabla {N})^ {\sharp}(\tilde{Q}_t(v))\, \vd l_t, \tilde{Q}_t(v)\r)\\
&=-\frac12 \l((\tr \nabla^2+\nabla_Z)\nabla \bd P_{T-t}f\r)(\tilde{Q}_t(v), \tilde{Q}_t(v))\, \vd t
\\
&\quad-\nabla\bd P_{T-t}f(R^{\sharp, \sharp}(\tilde{Q}_t(v),\tilde{Q}_t(v)))\, \vd t+\frac12 \bd P_{T-t}f\l((\bd ^*R-R(Z)+\nabla \Ric_Z)^{\sharp}(\tilde{Q}_t(v),\tilde{Q}_t(v))\r)\, \vd t\\
&\quad+ \nabla \bd P_{T-t}f\l((\nabla {N})^ {\sharp}(\tilde{Q}_t(v)), \tilde{Q}_t(v)\r)\, \vd l_t.
\end{align*}
Then using the definition of $W_t$, we calculate the quadratic covariation of $\bd P_{T-t}f$ and $W_t(v,v)$ as
\begin{align*}
\big[ {\rm D}(\bd P_{T-t}f), {\rm D} W_t(v,v)\big]&=\l[\nabla_{\ptr_t\vd B_t}\bd P_{T-t}f,R(\ptr_t\vd B_t,\tilde{Q}_t(v))\tilde{Q}_t(v)\r]\\
&=\tr\,\langle\nabla_{\displaystyle\bf.}\bd P_{T-t}f, R(\newdot,\tilde{Q}_t(v))\tilde{Q}_t(v)\rangle\,\vd t\\
&=\nabla \bd P_{T-t}f(R^{\sharp,\sharp}(\tilde{Q}_t(v),\tilde{Q}_t(v)))\, \vd t.
\end{align*}
According to the definition of $W_t(v,v)$, we have
\begin{align*}
(\bd &P_{T-t}f)({\rm D}W_t(v,v))+(\partial_t\bd P_{T-t}f)(W_t(v,v))\,\vd t\\
&=(\bd P_{T-t}f)\Big(R(\ptr_t\vd B_t, \tilde{Q}_t(v))\tilde{Q}_t(v)\\
&\quad -\frac12 (\bd^*R-R(Z)+\nabla \Ric_Z)^{\sharp}(\tilde{Q}_t(v), \tilde{Q}_t(v))\,\vd t\\
&\quad -\frac12 (\nabla^2N-R({N}))^{\sharp}(\tilde{Q}_t(v), \tilde{Q}_t(v))\,\vd l_t+\frac12 (\nabla {N})^{\sharp}(W_t(v, v))\,\vd l_t\Big)\\
&\quad -\frac{1}{2}(\tr \nabla^2+\nabla_Z)\bd P_{T-t}f(W_t(v,v))\,\vd t.
\end{align*}
We conclude that
\begin{align}\label{N-eq1}
  &(\nabla \bd P_{T-t}f)\l({\rm D}\tilde{Q}_t(v), \tilde{Q}_t(v)\r)+(\nabla \bd P_{T-t}f)\l(\tilde{Q}_t(v), {\rm D}\tilde{Q}_t(v)\r)+\partial_t(\nabla \bd P_{T-t}f)(\tilde{Q}_t(v), \tilde{Q}_t(v))\, \vd t\notag\\
  &\quad+\frac12 (\tr\nabla^2+\nabla_Z)(\nabla \bd P_{T-t}f)
    (\tilde{Q}_t(v), \tilde{Q}_t(v))\,\vd t
    +(\bd P_{T-t}f)({\rm D}W_t(v,v))+[ {\rm D} (\bd P_{T-t}f), {\rm D}W_t(v,v)]\notag\\
  &\quad+\partial_t(\bd P_{T-t}f)(W_t(v,v))\, \vd t+\frac12 
    (\tr \nabla^2+\nabla_Z)(\bd P_{T-t}f)(W_t(v,v))\,\vd t\notag\\
  &=-\frac12 (\bd P_{T-t}f)(\nabla^2 N-R(N))^{\sharp}(\tilde{Q}_t(v),\tilde{Q}_t(v))\,\vd l_t+\frac12 (\bd P_{T-t}f)((\nabla {N})^{\sharp}(W_t(v,v)))\,\vd l_t\notag\\
  &\quad+\frac12 \nabla \bd P_{T-t}f((\nabla {N})^{\sharp}(\tilde{Q}_t(v)),\tilde{Q}_t(v))\,\vd l_t+\frac12 
    \nabla \bd P_{T-t}f(\tilde{Q}_t(v), (\nabla {N})^{\sharp}(\tilde{Q}_t(v))\,\vd l_t.
\end{align}
On the other hand, for the terms in \eqref{Ito-Nt} related to the normal vector on the boundary, we have
\begin{align*}
&\nabla_N(\nabla \bd P_{T-t}f)(\tilde{Q}_t(v),\tilde{Q}_t(v))\,\vd l_t+\nabla_N(\bd P_{T-t}f)(W_t(v,v))\,\vd l_t\notag\\
&=-\nabla \bd P_{T-t}f\big((\nabla N)^{\sharp}(\tilde{Q}_t(v)),\tilde{Q}_t(v)\big)\,\vd l_t-
\nabla \bd P_{T-t}f(\tilde{Q}_t(v), (\nabla N)^{\sharp}(\tilde{Q}_t(v)))\,\vd l_t \notag\\
&\quad-\bd P_{T-t}f((\nabla^2 N+R(N))(\tilde{Q}_t(v), \tilde{Q}_t(v)))\,\vd l_t-
\bd P_{T-t}f((\nabla N)^{\sharp}(W_t(v,v)))\, \vd l_t.
\end{align*}
Combining this with \eqref{Ito-Nt} and \eqref{N-eq1}, we obtain
\begin{align*}
\vd M_t(v,v)
&\mequal\frac12 (\bd P_{T-t}f)\bigg((\nabla^2N+R({N}))(\tilde{Q}_t(v), \tilde{Q}_t(v))(X_t)\,\vd l_t-(\nabla {N})^{\sharp}(W_t(v, v))(X_t)\, \vd l_t\bigg)\\
&\quad-\frac12 (\bd P_{T-t}f)\l((\nabla^2 {N}+R({N}))(\tilde{Q}_t(v),\tilde{Q}_t(v))\r)(X_t)\, \vd l_t\\
&\quad +\frac12 (\bd P_{T-t}f)\l((\nabla {N})^{\sharp}(W_t(v, v))\r)(X_t)\, \vd l_t
=0.
\end{align*}
In other words, $M_t(v,v)$ is a local martingale. 

Let 
\begin{align*}
M_t^{\tilde{h}}(v,v)=\nabla \bd P_{T-t}f(\tilde{Q}_t(\tilde{h}(t)v), \tilde{Q}_t(v))+(\bd P_{T-t}f)(W_t^{\tilde{h}}(v,v)).
\end{align*}
From the definition of $W^{\tilde{h}}_t(v,v)$, resp.~$W_t(v,v)$, and in
view of the fact that $M_t(v,v)$ is a local martingale, we 
see that
\begin{align}\label{local-M1}
  M_t^{\tilde{h}}(v,v)&-\int_0^t(\nabla \bd P_{T-s}f)(\tilde{Q}_s(h_s
              v), \tilde{Q}_s(v))\,\vd s
\end{align}
is a local martingale as well.  Replacing in $M_t^{\tilde{h}}(v,v)$ the second
argument $v$ by $\tilde{h}(t)v$, we further get that also
\begin{align}\label{local-M2-1st}
  M_t^{\tilde{h}}(v,{\tilde{h}}(t)v)&-\int_0^t(\nabla \bd P_{T-s}f)(\tilde{Q}_s(h_s
                  v), \tilde{Q}_s(\tilde{h}(t)v))\,\vd  s\notag\\
                &-\int_0^t\nabla \bd P_{T-s}f(\tilde{Q}_s(h_sv), \tilde{Q}_s(\tilde{h}_sv))\,\vd s -\int_0^t(\bd P_{T-s}f)(W^{\tilde{h}}_s(v,h_sv))\,\vd s \notag \\
                &+ \int_0^t\int_0^s(\nabla \bd P_{T-r}f)(\tilde{Q}_r(h(r)
                  v), \tilde{Q}_r(h_sv))\,\vd r \,\vd s
\end{align}
is a local martingale. Note that $M_t^{\tilde{h}}(v,\tilde{h}(t)v)=M_t^{\tilde{h}}(v,v)\,\tilde{h}(t)$.
Exchanging the order of integration in the
last term shows that
\begin{align}
  &M_t^{\tilde{h}}(v,\tilde{h}(t)v)-\int_0^t(\nabla \bd P_{T-s}f)(\tilde{Q}_s(h_s
    v), \tilde{Q}_s(\tilde{h}(t)v))\,\vd s\notag\\
  &\qquad-\int_0^t\nabla \bd P_{T-s}f(\tilde{Q}_s(\tilde{h}_sv), \tilde{Q}_s(h_sv))\,\vd s -\int_0^t(\bd P_{T-s}f)(W^{\tilde{h}}_s(v,h_sv))\,\vd s \notag \\
  &\qquad+ \int_0^t(\nabla \bd P_{T-r}f)(\tilde{Q}_r(h(r)
    v), \tilde{Q}_r((\tilde{h}(t)-\tilde{h}(r))v))\,\vd r\notag\\
  &=M_t^{\tilde{h}}(v,\tilde{h}(t)v) -\int_0^t(\bd P_{T-s}f)(W^{\tilde{h}}_s(v,h_sv))\,\vd s -2\int_0^t\nabla \bd P_{T-s}f(\tilde{Q}_s(h_sv), \tilde{Q}_s(\tilde{h}_sv))\,\vd s\label{local-M2}
\end{align}
is a local martingale. Moreover, since $NP_{T-t}f(X_t)\1_{\{X_t\in \partial M\}}=0$ and  by the It\^{o} formula, we have
\begin{align}\label{Pt-martingale}
  P_{T-t}f(X_t)=P_Tf(x)+\int_0^t\bd P_{T-s}f(\ptr_s \,\vd   B_s).
\end{align}
The usual integration by parts yields 
\begin{align}\label{local-M2-1}
  \int_0^t(\bd P_{T-s}f)(W_s^{\tilde{h}}(v,h_sv))\,\vd s-P_{T-t}f(X_t)
  \int_0^t\langle W_s^{\tilde{h}}(v,h_sv), \ptr_s\,\vd   B_s \rangle
\end{align}
is a local martingale. 

On the other hand, from the It\^{o} formula and the commutation rule
 \eqref{commutation-1}, we obtain
 \begin{align*}
 \,\vd  (\bd P_{T-t}f(\tilde{Q}_t(v)))&=\nabla \bd P_{T-t}f(\ptr_t \,\vd B_t,\, \tilde{Q}_t(v) )-
 \frac12  \bd (\Delta P_{T-t}f) (\tilde{Q}_t(v))\,\vd t\\
 &\quad  +\frac12 (\tr \nabla ^2 \bd P_{T-t}f)(\tilde{Q}_t(v))\,\vd t+\frac 12  \nabla _N (\bd P_{T-t}f)(\tilde{Q}_t(v))\,\vd  l_t \\
 &\quad - \frac12 \bd P_{T-t}f(\Ric^{\sharp}(\tilde{Q}_t(v)))\,\vd t-\frac12 \bd P_{T-t}f((\nabla N)^{\sharp}(\tilde{Q}_t(v)))\,\vd l_t\\
 &=\nabla \bd P_{T-t}f(\ptr_t\,\vd B_t,\, \tilde{Q}_t(v) )+ \frac 12  \nabla _N (\bd P_{T-t}f)(\tilde{Q}_t(v))\,\vd  l_t\\
&\quad -\frac12 \bd P_{T-t}f((\nabla N)^{\sharp}(\tilde{Q}_t(v)))\,\vd l_t\\
 &=\nabla \bd P_{T-t}f(\ptr_t\,\vd B_t,\, \tilde{Q}_t(v) ) + \frac 12 \bd (N(P_{T-t}f)) (\tilde{Q}_t(v))\,\vd  l_t\\
 &=\nabla \bd P_{T-t}f(\ptr_t\,\vd B_t,\, \tilde{Q}_t(v) ).
 \end{align*}
It thus follows that
\begin{align*}
  \bd P_{T-t}f(\tilde{Q}_t(\tilde{h}(t)v))=\bd P_{T}f(v)+\int_0^t(\nabla \bd P_{T-s}f)
  (\ptr_s\,\vd   B_s, \tilde{Q}_s(\tilde{h}_s v))+\int_0^t\bd P_{T-s}f(\tilde{Q}_s(h_sv))\,\vd s.
\end{align*}
Integration by parts yields that
\begin{align}\label{local-M3}
  &\int_0^t(\nabla \bd P_{T-s}f)(\tilde{Q}_s(h_sv), \tilde{Q}_s(\tilde{h}_s v))
    \,\vd s-\bd P_{T-t}f(\tilde{Q}_t(\tilde{h}(t) v))\int_0^t\langle \tilde{Q}_s(h_sv),
    \ptr_s\,\vd   B_s\rangle \notag\\
  &\quad+\int_0^t\bd P_{T-s}f(\tilde{Q}_s(h_sv))\,\vd s \,\int_0^t\langle \tilde{Q}_s(h_sv)
    \ptr_s\,\vd   B_s\rangle
\end{align}
is also a local martingale. Concerning the last term in
\eqref{local-M3}, we note that
\begin{align*}
  &\int_0^t\bd P_{T-s}f(\tilde{Q}_s(h_sv))\,\vd s \,\int_0^t\langle \tilde{Q}_s(h_sv),\ptr_s\, \vd  B_s\rangle - \int_0^t \bd P_{T-s}f(\tilde{Q}_s(h_sv))\l(\int_0^s\langle \tilde{Q}_r(h(r)v),\ptr_r\,\vd   B_r\rangle\r)\,\vd s
\end{align*}
is a local martingale. Combining this with \eqref{local-M3} we
conclude that
\begin{align}\label{local-M3-1}
  &\int_0^t(\nabla \bd P_{T-s}f)(\tilde{Q}_s(h_sv), \tilde{Q}_s(\tilde{h}_s v))
    \,\vd s-\bd P_{T-t}f(\tilde{Q}_t(\tilde{h}(t) v))\int_0^t\langle \tilde{Q}_s(h_sv),
    \ptr_s\,\vd   B_s\rangle \notag\\
  &\quad+\int_0^t \bd P_{T-s}f(\tilde{Q}_s(h_sv))\int_0^s\langle \tilde{Q}_r(h(r)v),\ptr_r\,\vd   B_r\rangle\,\vd s
\end{align}
is a local martingale.  

Using the local martingales
\eqref{local-M2-1} and \eqref{local-M3-1} to replace the last two
terms in \eqref{local-M2}, we conclude that
\begin{align}\label{Martingale-1}
  &(\nabla \bd P_{T-t}f)(\tilde{Q}_t(\tilde{h}(t)v), \tilde{Q}_t(\tilde{h}(t) v))+(\bd P_{T-t}f)
    (W_t^{\tilde{h}}(v,\tilde{h}(t) v))  \notag\\
  &\quad -P_{T-t}f(X_t)\int_0^t\langle W_s^{\tilde{h}}(v,h_s v),\ptr_s\,\vd B_s \rangle \notag\\
  &\quad -2\bd P_{T-t}f(\tilde{Q}_t(\tilde{h}(t) v))\int_0^t\langle \tilde{Q}_s(h_sv),
    \ptr_s\,\vd B_s \rangle  \notag\\
  &\quad +2\int_0^t\bd P_{T-s}f(\tilde{Q}_s(h_sv)) \int_0^s\langle \tilde{Q}_r(h(r)v),\ptr_r\,\vd B_r\rangle \,\vd s 
\end{align}
is a local martingale as well.  On the other hand, by the product rule
for martingales, we have
\begin{align}\label{eqn}
  \l(\int_0^t\langle \tilde{Q}_s(h_sv),
    \ptr_s\,\vd B_s \rangle\r)^2-\int_0^t |\tilde{Q}_s(h_sv)|^2\,\vd s = 2\int_0^t \l(\int_0^s\langle \tilde{Q}_r(h(r)v),
    \ptr_r\,\vd B_r \rangle\r)\langle \tilde{Q}_s(h_sv),
    \ptr_s\,\vd B_s \rangle
\end{align}
which along with \eqref{Pt-martingale} implies that
\begin{align*}
  & P_{T-t}f(X_t)\l(\l(\int_0^t\langle \tilde{Q}_s(h_sv),
    \ptr_s\,\vd B_s \rangle\r)^2-\int_0^t |\tilde{Q}_s(h_sv)|^2\,\vd s\r) \\
  &\quad-2\int_0^t\bd P_{T-s}f(\tilde{Q}_s(h_sv)) \int_0^s\langle \tilde{Q}_r(h(r)v),\ptr_r\,\vd B_r\rangle \,\vd s
\end{align*}
is a local martingale.  Applying this observation to
~\eqref{Martingale-1}, we finally see that
\begin{align*}
  &(\nabla \bd P_{T-t}f)\big(\tilde{Q}_t(\tilde{h}(t)v), \tilde{Q}_t(\tilde{h}(t) v)\big)+(\bd P_{T-t}f)
    (W_t^{\tilde{h}}(v,\tilde{h}(t) v))  \notag\\
  &\ -2\bd P_{T-t}f(\tilde{Q}_t(\tilde{h}(t) v))\int_0^t\langle \tilde{Q}_s(h_sv),
    \ptr_s\,\vd B_s \rangle  \notag\\
  &\ -P_{T-t}f(X_t)\int_0^t\langle W_s^{\tilde{h}}(v,h_s v),\ptr_s\,\vd B_s \rangle \notag\\
  &\ +P_{T-t}f(X_t)\l(\l(\int_0^t\langle \tilde{Q}_s(h_sv),
    \ptr_s\,\vd B_s \rangle\r)^2-\int_0^t |\tilde{Q}_s(h_sv)|^2\,\vd s\r)
\end{align*}
is a local martingale. This completes the proof.
\end{proof}

With the Lemmas \ref{lem3} and \ref{lem4},  we are now in position to prove Theorem \ref{local-them3}.

\begin{proof}[Proof of Theorem \ref{local-them3}]
Let  $h^{\eps}_s=0$
for $s\geq (T-\eps)\wedge \tau_k$. Let $B_k:=\{x: \rho_o(x)\leq k\}$ for $k\geq 1$.
By the strong Markov property, the boundedness of $P{\bf.}f$ on $[\eps,T]\times B_k$
and the boundedness of $|\bd P{\bf.}f|$ and $|\Hess_{P{\bf.}}f|$  on
$[\eps,T]\times B_k$ for $f\in \mathcal{B}_b(M)$, it follows from Lemma \ref{lem3} that
\begin{align*}
(\nabla \bd P_{T}f)(v,v)=&-\E\l[f(X_T^x)\int_0^{(T-\eps)\wedge \tau_k}\langle W_s^{\tilde{h^{\eps}}}({h}^{\eps}_sv,v), \ptr_s \vd B_s\rangle\r]\\
&+\E\l[f(X_T^x)\l(\l(\int_0^{(T-\eps)\wedge \tau_k}\langle \tilde{Q}_s({h}^{\eps}_sv),
      \ptr_s \vd  B_s \rangle\r)^2-\int_0^{(T-\eps)\wedge \tau_k} |\tilde{Q}_s({h}^{\eps}_sv)|^2  \,\vd s\r)\r].
\end{align*}
Letting $\eps\downarrow 0$,  we have
\begin{align*}
(\nabla \bd P_Tf)(v,v)&=-\E\l[f(X_T^x)\int_0^{T\wedge \tau_k}\langle W_s^{\tilde{h}_s}(h_s v,  v),\ptr_s\vd B_s \rangle\r]\\
&\quad +\E\l[f(X_T^x)\l(\l(\int_0^{T\wedge \tau_k}\langle \tilde{Q}_s(h_sv), \ptr_s\vd B_s \rangle \r)^2-\int_0^{T\wedge \tau_k} | \tilde{Q}_s(h_sv)|^2\,\vd s\r)\r].
\end{align*}
By Lemma \ref{lem4} and the observation that there exists a constant $c>0$ such that
\begin{align*}
&\E\l [\sup_{t\in [0,T]}\Big|\l(\l(\int_0^{t\wedge \tau_k}\langle \tilde{Q}_s(h_sv), \ptr_s\vd B_s \rangle \r)^2-\int_0^{t\wedge \tau_k} | \tilde{Q}_s(h_sv)|^2\,\vd s\r)\Big|\r] \\
&\leq c\, \E\l[\int_0^T\e^{\int_0^s K^-(X_s)\, \d s +\int_0^s \sigma^-(X_s) \, \d l_s} h_s^2\, \d s\r],
\end{align*}
we complete the proof  by Fatou's lemma.
\end{proof}

\subsection{Global Hessian estimates of the semigroup}
In this subsection, we continue the discussion on explicit global
estimates for $\Hess_{P_tf}$ under suitable conditions.  \medskip

For $\eps>0$, let 
\begin{align*}
\mathcal{D}_{\eps}:=\{\phi\in C_b^2(M):\, \inf \phi=1,\ N\log\phi \geq \sigma^-+\eps\}.
\end{align*}

\fbox{\,\parbox{0.95 \textwidth}{
    \noindent{\bf (B) }
    The functions $K, \sigma$ in \eqref{CD-1} and
    $\alpha, \beta, \gamma$ in \eqref{CD-2} are constant, and there exists $\phi\in \mathcal{D}_{\eps}$ for some $\eps>0$ such that 
\begin{align}\label{Kphiq}
K_{\phi,q}=\sup_{x\in M} \big\{-L\log \phi+2q |\nabla \log \phi|^2  \big\}<\infty
\end{align}
for some positive constant $q>1$.  }}



  \medskip
  \noindent By  \cite[Section 3.2]{Wbook14}, such $\phi$ can be constructed if $\pp M$ has strictly positive injectivity radius,   the sectional curvature of $M$ being bounded above and $Z$ bounded. In particular, if the manifold is compact, this condition are met automatically.
  Under the global bounds of condition {\bf (B)},  it holds that
\begin{align*}
\Ric_Z+L\log \phi-2|\nabla \log\phi|^2\geq K-K_{\phi},
\end{align*}
where we write $K_{\phi}:=K_{\phi,1}$ for simplicity. By  \cite[Theorem 2.2]{CTT18b}, we obtain
\begin{align}\label{HP1}
\|\nabla P_tf\|_{\infty}\leq \|\phi\|_{\infty}\|\nabla f\|_{\infty}\e^{-(K-K_{\phi,1})t},\quad t>0,
\end{align}
which implies that $|\nabla P{\bf.}f|$ is bounded on $[0,T]\times M$ for $f\in C^1_b(M)$.
  
  Next  local Bismut formulae, as the one
  in Theorem \ref{local-them3} for
  $\Hess_{ P_{t}f}$, permit us to show that for
  any $\eps>0$,
  \begin{equation}\label{bounded-eps}
  |\Hess_{P{\displaystyle\bf.}f}|\  \text{is bounded on $[\eps, T]\times M$}.
  \end{equation}
   This requires for $x\in M$ and a given relatively compact open
  neighbourhood $D$ of $x$, the construction of an adapted real process $h_t$
  such that $h_t=0$ for $t\geq T\wedge\tau_D$ and  
  $\int_0^{T\wedge\tau_D}h_t\,\vd t=-1$ with the property that 
$$\E\Big[\int_0^{T}h_t^{2p}\,\vd t\Big]<\infty,\ \mbox{and}\ \  \sup _{x\in M}\E^x[ \e^{q(\sigma^-+\eps)l_t}]<\infty$$ 
for $1/p+1/q=1$ and $p,q>1$,
 where $\tau_D:=\inf\{t\geq0\colon X_t^x\not\in D\}$ denotes the first exit time of $D$, see estimate \eqref{Hess_P_Tf_Estimate}.
  In Remark \ref{Remark_eps} below we briefly sketch the construction of processes $h$ 
with the required properties. Before this, let us introduce the conformal change of the metric such that the boundary under the
new metric is convex.

\begin{remark}[Conformal change of the metric]\label{conformal-metric}
We start with a conformal change of the metric $g$.  Since $\phi \in \mathcal{D}_{\eps}$, we have $\II\geq \sigma\geq -(\sigma^-+\eps) \geq  -N\log \phi$ and the boundary $\partial M$
is convex under the metric $g':=\phi^{-2}g$. Let $\Delta'$ and $\nabla'$ be the Laplacian and gradient operator
associated to the metric $g'$. Then
\begin{align*}
L=\phi^{-2}(\Delta'+\phi^2(Z+(d-2)\nabla \log \phi))=\phi^{-2}(\Delta'+Z')
\end{align*}
where $Z':=\phi^2(Z+(d-2)\nabla \log \phi)$.  Let $\rho'(x,y)$ be the geodesic distance from $x$ to $y$ with respect to the metric $g'$ on $M$.

Furthermore, let 
\begin{align*}
&U_i= \phi^{-1}(\gamma(s)) P'_{\gamma(0),\gamma (s)} V_i, \\
&J_i(s)=f(s) U_i, \quad 1\leq i\leq d, \
\end{align*}
where  $\{V_i\}_{i=1}^d$ is a $g'$-orthonormal basis of 
$T_xM$, $P'_{\gamma(0),\gamma(s)}$ denotes parallel displacement from $x$ to $y$ with respect to the metric $g'$ and  $f(s)=1\wedge \frac{s}{\rho(x,y)\wedge 1}$. Then $J_i(0)=0$ and $J_i(\rho')=\phi^{-1}(y)P'_{x,y}V_i,\ 1\leq i\leq d$,
\begin{align}\label{Laplacian-comparison}
&\phi^{-2}(\Delta'+Z')\rho'(x,\newdot)(y)\notag\\
&\leq \sum_{i=1}^d \int_0^{\rho'}\Big \{ (|\nabla'_{\dot{\gamma}}J_i|')^2-\langle R'(\dot{\gamma}, J_i)J_i, \dot{\gamma} \rangle'\Big\}(s)\, \vd s+ \phi^{-2}(y) Z'\rho'(x,\newdot)(y) \notag \\
&  \leq \sum_{i=1}^d \int_0^{\rho'} \Big\{ f'(s)^2\phi^{-2}(\gamma(s)) +f(s)^2(|\nabla'_{\dot{\gamma}}U_i|')^2-f(s)^2\langle R'(\dot{\gamma}, U_i)U_i, \dot{\gamma} \rangle'\Big\}(s)\, \vd s + \phi^{-2}(y) Z'\rho'(x,\newdot)(y).
\end{align}
On the other hand,
\begin{align}\label{phiZ}
&\phi^{-2}(y) Z'\rho'(x,\newdot)(y) \notag \\
& =\int_0^{\rho'} \frac{\vd}{\vd s} \l\{ f(s)^2\phi^{-2}(\gamma(s)) \langle Z'(\gamma(s)), \dot{\gamma} (s) \rangle' \r\}\, \vd s \notag \\
&= 2\int_0^{\rho'} f'(s)f(s)\phi^{-2}(\gamma(s)) \langle Z'(\gamma), \dot{\gamma}  \rangle'(s)+ f(s)^2\frac{\vd}{\vd s} \l\{\phi^{-2}(\gamma) \langle Z'(\gamma), \dot{\gamma}  \rangle' \r\}(s)\, \vd s  \notag \\
&=2\int_0^{\rho'} f'(s)f(s)\phi^{-2}(\gamma(s)) \langle Z'(\gamma), \dot{\gamma}  \rangle'(s)\,\vd s \notag  \\
&\quad +
\int_0^{\rho'}f(s)^2 \phi ^{-2}(\gamma(s))\langle (\nabla'_{\dot{\gamma}}Z')\circ \gamma, \dot{\gamma} \rangle'(s)\, \vd s \notag \\
&\quad -2\int_0^{\rho'} f(s)^2 \phi^{-2}(\gamma(s)) \langle \nabla \log \phi (\gamma(s)),\dot{\gamma}(s) \rangle \langle Z'(\gamma (s)), \dot{\gamma}(s) \rangle'\, \vd  s
\end{align}
Note that $|\dot{\gamma}|=\phi$.
We then conclude from \eqref{Laplacian-comparison} and \eqref{phiZ} that

\begin{align*}
&\phi^{-2}(y)(\Delta'+Z')\rho'(x,\newdot)(y)\\
&\leq  - \int_0^{\rho'} f(s)^2\phi^{-2}(\gamma(s)) \Big\{ (\Ric^Z)'(\dot{\gamma}(s),\dot{\gamma}(s))+(d-4)\langle \nabla \log \phi, \dot{\gamma}(s)  \rangle ^2+ 2 \langle Z, \dot{\gamma}(s) \rangle \langle \nabla \log \phi, \dot{\gamma}(s) \rangle  \Big \}\, \vd s \\
&\quad +2\int_0^{\rho'} f'(s)f(s)\phi^{-2}(\gamma(s)) \langle Z'(\gamma), \dot{\gamma}  \rangle'(s)\,\vd s \\
&\quad +
d\int_0^{\rho'}f'(s)^2 \phi ^{-2}(\gamma(s))\, \vd s \\
&\leq -(K-K_{\phi}) \rho'(x,y)+\frac{2}{\rho'(x,y)\wedge 1}\int_0^{\rho'(x,y)\wedge 1}\phi^{-2}(\gamma(s)) \big|\langle Z(\gamma(s)), \dot{\gamma}(s)  \rangle+(d-2) \langle \nabla \log \phi,\dot{\gamma}(s) \rangle\big |\,\vd s \\
&\quad +d \int_0^{\rho'(x,y)\wedge 1}  \frac{1}{(\rho'(x,y)\wedge 1)^2} \phi^{-2}(\gamma(s))\, \vd s \\
& \leq -(K-K_{\phi}) \rho'(x,y)+ 2\sup_{z\in B'(x,1)}(|Z|+(d-2)|\nabla \phi|)(z)+ \frac{d}{\rho'(x,y)\wedge 1}.
\end{align*}

The next step is to check that  for $\alpha>0$,
\begin{align}
&\sup_{x\in M}\E^x\big[ \e^{\frac{\alpha}{2}\sigma ^- l_t} \big]<\sup_{x\in M}\E^x\big[ \e^{\frac{\alpha}{2}(\sigma ^- +\eps) l_t} \big]<\|\phi\|_{\infty}^{\alpha} \exp {\Big(\frac{\alpha}{2}K_{\phi,\alpha} t \Big)} <\infty, \label{add-eq1}\\
& \sup_{x\in M}\E^x\big[ \e^{(\sigma ^-+\eps) l_t} \big]<\|\phi\|_{\infty} \exp {\Big(K_{\phi} t \Big)} <\infty, \label{add-eq2}
\end{align}
 where 
\begin{align*}
K_{\phi, \alpha}=\sup_{M} \l\{-L \log \phi +2\alpha|\nabla \log \phi|^2\r\},
\end{align*}
and  $K_{\phi}:=K_{\phi,1}$ for simplicity. 
By It\^{o}'s formula,
\begin{align*}
\vd \phi^{-\alpha}(X_t)&= \langle \nabla \phi^{-\alpha}(X_t), \ptr_t\vd B_t \rangle+\frac{1}{2}L\phi^{-\alpha}(X_t)\, \vd t+\frac{1}{2}N\phi^{-\alpha}(X_t)\, \vd l_t \\
 & \leq \langle \nabla \phi^{-\alpha}(X_t), \ptr_t\vd B_t \rangle-\alpha \phi^{-\alpha}(X_t)\l(-\frac{1}{2}K_{\phi,\alpha}\, \vd t+\frac{1}{2} N\log \phi(X_t)\, \vd l_t\r)  \\
 & \leq \langle \nabla \phi^{-\alpha}(X_t), \ptr_t\vd B_t \rangle-\alpha \phi^{-\alpha}(X_t)\l(-\frac{1}{2}K_{\phi,\alpha}\, \vd t+\frac{1}{2} (\sigma^-+\eps)\, \vd l_t\r), 
\end{align*}
then
\begin{align*}
\phi^{-\alpha}(X_t) \exp \l(-\frac{\alpha}{2}K_{\phi,\alpha} t+\frac{\alpha}{2}(\sigma^-+\eps)\, l_t\r)
\end{align*}
is a local submartingale. 
 Therefore,
by Fatou's lemma and taking into account that $\phi\geq 1$, we get
\begin{align*}
\E \l[\phi^{-\alpha}(X_t) \exp \l(-\frac{\alpha}{2}K_{\phi,\alpha}t+\frac{\alpha}{2}(\sigma^-+\eps)\, l_t\r)\r]\leq 1,
\end{align*}
which proves \eqref{add-eq1} and \eqref{add-eq2}.

\end{remark}

\begin{remark}[Construction of $h$]\label{Remark_eps}
  Let $D=B'(x,k)$ where $B'(x,k):=\big\{y\in M\colon \rho'(x,y)\leq k\big\}$ for some $k>0$.
  We search for an adapted real process $h=h_k$  satisfying  $\int_0^t (h_k)_s\, \vd s=-1$ for $t\geq T\wedge \tau_k$ and
  $$\E^x\l[\int_0^T(h_k^{2p})_s\, \vd s\r]<\infty$$
  where $\tau_k$ is the first exit time from $B(x,k)$. To $h_k$ we then consider
  $$(\tilde h_k)_t=1+\int_0^t(h_k)_s\,\vd s$$
  so that $(\tilde h_k)_0=1$ and $(\tilde h_k)_t=0$ for $t\geq T\wedge \tau_k$.
For $k>0$ let $$\theta_k(p)=\cos \l(\frac{\pi \rho'(x, p)}{2k}\r), \ \ p\in B(x,k).$$
Then set $(\tilde{h}_k)_s=(\tilde{h}\circ \ell_k)_s$ where a function $\tilde{h}\in C^1([0,T])$ is chosen so that
$\tilde{h}(0)=1, \ \tilde{h}(T)=0$  with $(\tilde{h})'=h$ and 
$$\ell_k(s)=\int_0^s \theta_k^{-2}(X_r(x))\1_{\{r<\sigma_k(T)\}}\, \vd r,$$
$$\sigma_k(s)=\inf\l\{r\geq 0:\ \int_0^r \theta_k^{-2}(X_u(x))\, \vd u \geq s \r\}.$$
This construction is due to \cite{TW98}, the claim follows from
\cite{TW11,TW98}, see  the proof of \cite[Lemma 2.1]{CTT18}
for the details.  For this $\tilde{h}_k$, we have 
\begin{align*}
\E \l[\int_0^{t\wedge \tau_k} {h}_k^{2p}(s)\, \vd s\r]&=\E \l[ \int_0^{\sigma(t)} (h\circ \ell_k)^{2p}(s)\theta_k^{-4p}(X_s(x))\, \vd s \r]\\
&=\int_0^t h^{2p}(s)\E \big[\theta_k^{-4p+2}(X_s'(x)) \big]\,\vd s
\end{align*}
where $X'(x)$ denotes the diffusion  starting at $x$ with generator $\frac{1}{2}\theta_k^2L$ which 
almost surely doesn't not exit $B'(x,k)$ by \cite[Proposition 2.3.]{TW98}. To estimate the integration we use
\begin{align*}
\frac{1}{2}\theta_k^2 L\theta_k^{-4p+2}=(2p-1)\theta_k^{-4p+2}\l[\frac{4p-1}{2}|\nabla \theta_k|^2-\theta_k L \theta_k\r]
\end{align*}
to obtain, via Ito's formula, Gronwall's lemma and the fact  $N\rho'(x,\cdot)\leq 0$, that 
\begin{align*}
\E[\theta_k^{-4p+2}(X_s'(x))]\leq \theta_k(x)^{-4p+2} \e^{c(\theta_k)s},
\end{align*} 
where
\begin{align*}
c(\theta_k)=(2p-1)\sup_{B'(x,k)}\l\{\frac{(4p-1)}{2}|\nabla \theta_k|^2-\theta_k L\theta_k \r\}.
\end{align*}
Using $\theta_k(x)=1$ and taking 
\begin{align*}
\tilde{h}(t)=1-\frac{c(\theta_k)}{p(1-\e^{-c(\theta_k)T/p})}\int_0^t\e^{-c(\theta_k)r/p}\, \vd r
\end{align*}
we obtain
\begin{align*}
\E \l[\int_0^{T\wedge \tau} {h}_k^{2p}(s)\, \vd s\r]&\leq \int_0^T \l(\frac{c(\theta_k)}{p}\r)^{2p}\frac{\e^{-c(\theta_k)s}}{(1-\e^{-c(\theta_k)T/p})^{2p}}\, \vd s  \\
&\leq \l(\frac{c(\theta_k)}{p}\r)^{2p} \frac{T}{(1-\e^{-c(\theta_k)T/p})^{2p}} \leq \frac{\e^{2c(\theta_k)T}}{ T^{2p-1}}. 
\end{align*}

Indeed, according to the definition of $\theta_k$, we have 
\begin{align*}
|\nabla \theta_k| \leq \frac{\pi}{2k},
\end{align*}
and by the Laplacian comparison theorem
\begin{align*}
-(\theta_k L\theta_k) (p)& \leq \cos \l(\frac{\pi \rho'(x, p)}{2k}\r) \sin \l(\frac{\pi \rho'(x, p)}{2k}\r)\frac{\pi}{2k} L\rho'(x,\newdot)(p)+\cos  \l(\frac{\pi \rho'(x, p)}{2k}\r)^2 \frac{\pi^2}{4k^2} \\
&\leq \frac{\pi^2 \rho'(x,p)}{4k^2}\l(c(x)+\frac{d}{\rho'(x,p)}+(K-K_{\phi})\rho'(x,p)\r)+ \frac{\pi^2}{4k^2} \\
&\leq \frac{c(x)\pi^2}{4k}+\frac{(d+1)\pi^2}{4k^2}+\frac{(K-K_{\phi})\pi^2}{4},\qquad  \rho'(x,p)\leq k
\end{align*}
for some constant $c(x)>0$, where
\begin{align*}
\Ric_Z+L\log \phi-2|\nabla \log \phi|^2\geq K-K_{\phi}.
\end{align*}
We then conclude that
$$c(\theta_k)\leq (2p-1) \l(\frac{c(x)\pi^2}{2k}+\frac{(2d+4p+1)\pi^2}{4k^2}+\frac{(K-K_{\phi})\pi^2}{2}\r).$$
Then by the local version of the Bismut type Hessian formula, we have
\begin{align*}
| \Hess_{ P_{T}f } |(x) & \leq 3\e^{K^-T} \|f\|_{\infty} \l[\E\int_0^T h_k^2(s) \e^{\sigma^- l_s}\, \vd s\r]^{1/2} \Bigg\{\l [(3+\sqrt{10})\alpha +\frac{\beta}{2}\r]\l(\E \int_0^{T} \e^{ \sigma^-  l_s} \, \vd s \r)^{1/2}\\
&\qquad + \frac{\gamma}{2}\l(\E\int_0^{T}\e^{ \sigma^-  l_s} \, \vd l_s\r)^{1/2}+\frac{2}{3}\l(\E\int_0^T h_k^2(s) \e^{\sigma^- l_s}\, \vd s\r)^{1/2} \Bigg\}\\
&\leq 3\e^{K^-T} \|f\|_{\infty} \l(\|\phi\|_{\infty} \e^{ K_{\phi, q}T}T^{1/q} \l(\frac{\e^{2c(\theta_k)T}}{ T^{2p-1}}\r)^{1/p}\r)^{1/2} \Bigg\{\l((3+\sqrt{10})\alpha +\frac{\beta}{2}\r)\l(\|\phi\|_{\infty} \e^{ K_{\phi}T}T\r)^{1/2}\\
&\qquad + \frac{\gamma}{2}\l(\frac{\|\phi\|_{\infty}\e^{K_{\phi}T}}{\sigma^-+\epsilon}\r)^{1/2}+\frac{2}{3}\l[\|\phi\|_{\infty} \e^{ K_{\phi, q}T}T^{1/q} \l(\frac{\e^{2c(\theta_k)T}}{ T^{2p-1}}\r)^{1/p}\r]^{1/2} \Bigg\}.
\end{align*}
When the manifold is non-compact, letting $k$ tend to $\infty$ yields
\begin{align*}
| \Hess_{ P_{T}f } |(x) &\leq 3\e^{K^-T}\|\phi\|_{\infty} \|f\|_{\infty} \l( \e^{ K_{\phi,q}T+\frac{2p-1}{p}(K-K_{\phi})\pi^2}T^{-1}\r)^{1/2} \Bigg\{\l ((3+\sqrt{10})\alpha +\frac{\beta}{2}\r)\l(\e^{ K_{\phi}T}T\r)^{1/2}\\
&\qquad + \frac{\gamma}{2}\l(\frac{\e^{K_{\phi}T}}{\sigma^-+\epsilon}\r)^{1/2}+\frac{2}{3}\l( \e^{ K_{\phi,q}T+\frac{2p-1}{p}(K-K_{\phi})\pi^2}T^{-1}\r)^{1/2} \Bigg\}<\infty
\end{align*}
for $T>0$.  
\end{remark}

\begin{theorem}\label{th4}
  Assume that condition {\bf (B)} holds.
  Let $h$  be a non-positive and adapted process satisfying  $\int_0^T h_s\,\vd s=-1$  and
  $$\E^x\l[\int_0^T(h^2_s+\tilde{h}^2_s)\e^{\sigma^-l_s}\, \vd s\r]<\infty,$$ 
where $\tilde h_t=1+\int_0^t h_s\,\vd s$. Then,   for
  $f\in \mathcal{B}_b(M)$ and $v\in T_xM$,
  \begin{align*}
    \Hess_{ P_{T}f} (v, v)(x)=
    &- \E^x \l[f(X_T)\int_0^T\langle W_s^{\tilde{h}}(v, h_s v),\ptr_s\vd B_s \rangle \r] \notag\\
    &+ \E^x \l[ f(X_T)\l(\l(\int_0^T\langle \tilde{Q}_s(h_sv),
      \ptr_s \vd  B_s \rangle\r)^2-\int_0^T |\tilde{Q}_s(h_sv)|^2  \,\vd s\r) \r].
  \end{align*} 
Moreover,  for $T>0$  and $f\in \mathcal{B}_b(M)$,
\begin{align*}
|\Hess_ {P_Tf}|(x)\leq & \l(\alpha+\frac{\beta}{2}\sqrt{T}+\frac{2}{T}\r)\e^{K^-T}\E^x[\e^{\sigma  l_T}](P_{T}f^2)^{1/2}\\
&+\frac{\gamma }{2\sqrt{T}}\e^{K^-T}\E^x\l[\e^{\sigma l_T}\r]^{1/2}\l[\E^x\l(\int_0^T\e^{\frac{1}{2}\sigma l_s}\, \vd l_s\r)^2\r]^{1/2}(P_{T}f^2)^{1/2}.
\end{align*}
\end{theorem}

\begin{proof}
For the adapted process $h$, we see from Lemma \ref{lem4} that 
\begin{align*}
& \E\l[\sup_{t\in [0,T]} \Big | \int_0^t h_s \langle W_s^{\tilde{h}}(v,v), \, \ptr_s \vd B_s  \rangle \Big| \r]  \\
&\leq  3 \e^{K^-T}\l(\E^x\int_0^{T}  \e^{ \si^-  l_s}h^2_s \, \d s \r) ^{1/2} \Bigg\{ \alpha (3+\sqrt{10})\l(\E\int_0^T \e^{\sigma^-\,  l_s} \tilde{h}_s^2\, \vd s\r)^{1/2}\\
&\qquad + \frac{\beta}{2}\l(\E\int_0^T \e^{\sigma^-\,  l_s} \tilde{h}_s^2 \, \vd s\r)^{1/2} + \frac{\gamma}{2}\l(\E\int_0^T\e^{\sigma^-\,  l_s}\tilde{h}_s^2\, \vd l_s\r)^{1/2}\Bigg\} <\infty,
\end{align*}
and 
\begin{align*}
& \E \l[\Bigg|\l(\int_0^T\langle \tilde{Q}_s(h_sv),
      \ptr_s \vd  B_s \rangle\r)^2-\int_0^T |\tilde{Q}_s(h_sv)|^2  \,\vd s\Bigg|\r] \\
&\leq 2 \e^{K^-T} \E\l( \int_0^T  \e^{\sigma^-\,  l_s} \tilde{h}_s^2 \, \vd s\r)<\infty.
\end{align*}
Moreover, by \eqref{bounded-eps} both $|\nabla P\!{\bf.}f|$ and $|\Hess_{P{\displaystyle\bf.}f}|$ are bounded on $[\eps, T]\times M$.
We complete the proof  by following the steps as in the  proof of Theorem \ref{local-them3} 
to obtain from \eqref{Loc_mart_Hessian}
 \begin{align}\label{Est-eps}
    \Hess_{ P_{T}f} (v, v)(x)=
    &- \E^x \l[f(X_T)\int_0^T\langle W_s^{\tilde{h}}(v, h_s v),\ptr_s\vd B_s \rangle \r] \notag\\
    &+ \E^x \l[ f(X_T)\l(\l(\int_0^T\langle \tilde{Q}_s(h_sv),
      \ptr_s \vd  B_s \rangle\r)^2-\int_0^T |\tilde{Q}_s(h_sv)|^2  \,\vd s\r) \r].
 \end{align}
 Indeed, using the mentioned boundedness on $[\eps, T]\times M$, we get \eqref{Est-eps} first
 for $f$ replaced by $P_\eps f$ and from this \eqref{Est-eps} is obtained by letting $\eps$ tend to zero.
In particular, letting $h(s)=-\frac{1}{T}$ when $s\in [0,T]$ and $\tilde{h}(s)=\frac{T-s}{T}$ for $s\in [0,T]$, then
  \begin{align*}
|&\Hess_{P_Tf}|(x) \\
 &\leq (P_T|f|^2)^{1/2} \l[\E\l(\int_0^T\langle W_s^{\tilde{h}}(v, h(s) v),\ptr_s\vd B_s \rangle\r)^2\r]^{1/2} \\
&\qquad + 2 (P_T|f|^2)^{1/2} \int_0^T \e^{\sigma l_s -Ks}h^2(s)\,\vd s \\
&\leq \frac{1}{T} (P_T|f|^2)^{1/2}  \l(\E \l[\int_0^{T}\Big|\tilde{Q}_{s}\int_0^{s}\tilde{Q}_r^{-1} R(\ptr_r \vd B_r, \tilde{Q}_r(\tilde{h}(r)v))\tilde{Q}_r(v)\Big|^2\, \vd s \r]\r)^{1/2} \notag \\
&\quad+\frac{1}{2T} (P_T|f|^2)^{1/2} \l(\E\l[\int_0^{T}   \Big|\tilde{Q}_{s}\int_0^{s} \tilde{Q}_r^{-1} (\bd^*R-R(Z)+\nabla \Ric_Z)^{\sharp}(\tilde{Q}_r(\tilde{h}(r)v), \tilde{Q}_r(v))\,\vd r \Big|^2\, \vd s \r]\r)^{1/2} \notag\\
&\quad+ \frac {1}{2T}   (P_T|f|^2)^{1/2} \l( \E\l[\int_0^{T}  \Big|\tilde{Q}_{s}\int_0^{s} \tilde{Q}_r^{-1} (\nabla^2N-R(N))^{\sharp}(\tilde{Q}_r(\tilde{h}(r)v), \tilde{Q}_r(v))\,\vd l_r \Big|^2\, \vd s\r]\r)^{1/2} \\
&\quad + \frac{2}{T^2} (P_T|f|^2)^{1/2} \e^{K^-T}\int_0^T\E [\e^{\sigma l_s }]\,\vd s \\
&\leq  \l(\alpha+\frac{\beta}{2}\sqrt{T}+\frac{2}{T}\r)\e^{K^-T}\E^x[\e^{\sigma  l_T}](P_T|f|^2)^{1/2}\\
&\qquad +\frac{\gamma }{2\sqrt{T}}\e^{K^-T}\E^x\l[\e^{\sigma l_T}\r]^{1/2}\l[\E^x\l(\int_0^T\e^{\frac{1}{2}\sigma l_s}\, \vd l_s\r)^2\r]^{1/2}(P_T|f|^2)^{1/2}.\qedhere
\end{align*}
\end{proof}

\begin{corollary}\label{cor-global1}
 Assume that condition {\bf (B)} holds with $\sigma=\gamma=0$. Then
\begin{align*}
\l|\Hess _ {P_Tf}\r|\leq  \l(\alpha+\frac{\sqrt{T}}{2}\beta+\frac{2}{T}\r)\e^{K^-T}(P_T|f|^2)^{1/2}
\end{align*}
for $T>0$ and $f\in \mathcal{B}_b(M)$.
\end{corollary}

\begin{proof} This is a direct consequence of the estimate in Theorem \ref{th4}.
 \end{proof}
 
 \begin{remark} Note that the condition $-\nabla N \geq 0$, i.e. $\sigma=0$ 
implies  $\II\geq 0$.
It has been proved in \cite{Wbook14} that then
\begin{align}\label{GP1}
|\nabla P_tf |(x) \leq  \frac{ \e ^{K^-t}}{\sqrt{t}} (P_t|f|^2)^{1/2}<\infty
\end{align}
 for any $t>0$, see also \cite{Cheng_Thalmaier_Thompson:2018}.  
\end{remark}

\section{Hessian  formula with gradient terms}
The main theorem in this section relies on the fact that  under {\bf (B)}, along with suitable conditions,
the local martingale $M_t$ defined in \eqref{M-martingale}  is a true 
martingale. This fact will be exploited for further applications.

\begin{theorem}\label{Them-hessian-gradient}
  Assume that condition {\bf (B)} holds.
  For $T>0$ let $h\in C([0,T])$  such that $\int_0^T h(t)\, \vd t=-1$. Then, for $v\in T_xM$
  and $f\in C_b^1(M)$ such that
  $|\nabla P{\bf.}f|$ is bounded,
 \begin{equation}\label{Hessian-formula2}
    \Hess_{P_Tf}(v,v)=\E\l[-\bd f(\tilde{Q}_T(v))\int_0^T\langle \tilde{Q}_s(h(s)v),
    \ptr_s\vd  B_s\rangle+\bd f( W_T^{\tilde{h}}(v,v))\r]
  \end{equation}
  where $\tilde{h}(t)=1+\int_0^t h(s)\, \vd s$. Moreover,
\begin{align*}
|\Hess_{P_Tf}|& \leq  \l(\alpha \sqrt{T}+\frac{\beta}{2} T+\frac{1}{\sqrt{T}} \r) \E\l[\e^{\sigma^- l_T}\r]\e^{K^-T} \|\nabla f\|_{\infty} \\
&\quad  + \frac{\gamma}{2} \E \l[\e^{\frac{1}{2}\sigma^- l_T}\int_0^T\e^{\frac{1}{2}\sigma^- l_s}\, \vd l_s\r]\e^{K^-T} \|\nabla f\|_{\infty}.
\end{align*}
\end{theorem}

\begin{proof}
Recall that by \eqref{local-M1}
\begin{align*}
 \nabla \bd P_{T-t}f(\tilde{Q}_t(\tilde{h}(t)v), \tilde{Q}_t(v))+(\bd P_{T-t}f)(W_t^{\tilde{h}}(v,v))&-\int_0^t(\nabla \bd P_{T-s}f)(\tilde{Q}_s(h(s)
              v), \tilde{Q}_s(v))\,\vd  s
\end{align*}
is a local martingale. On the other hand,  we know from from the proof to Lemma \ref{lem3} that 
\begin{align*}
  &\int_0^t(\nabla \bd P_{T-s}f)(\tilde{Q}_s(h(s)v), \tilde{Q}_s(v))
    \,\vd s-\bd P_{T-t}f(\tilde{Q}_t( v))\int_0^t\langle \tilde{Q}_s(h(s)v),
    \ptr_s \vd   B_s\rangle
\end{align*}
is a local martingale as well. 
We conclude that
\begin{align*}
\nabla \bd P_{T-t}f(\tilde{Q}_t(\tilde{h}(t)v), \tilde{Q}_t(v))+(\bd P_{T-t}f)(W_t^{\tilde{h}}(v,v))&-\bd P_{T-t}f(\tilde{Q}_t( v))\int_0^t\langle \tilde{Q}_s(h(s)v),
    \ptr_s \vd   B_s\rangle
\end{align*}
is a local martingale.  
 As $\|R\|_{\infty}<\infty$, $\Ric\geq K$ for some constant $K$,
and   $-\nabla N \geq \sigma$ for some non-positive constant $\sigma$,
\begin{align*}
\big\|\bd ^*R+\nabla \Ric_Z^{\sharp}-R(Z)\big\|_{\infty}<\infty\quad  \mbox{and}\quad  \big\|\nabla^2 N+R(N)\big\|_{\infty}<\infty,
\end{align*}
we first get
\begin{align*}
&\sup_{s\in [0,T]}\E|\tilde{Q}_s(v)|^2\leq\sup_{s\in [0,T]}\e^{-Ks}\E[\e^{\sigma^- l_T}]<\infty,
\end{align*}
and then
\begin{align}\label{W-2}
&\E\l[|W_{t\wedge \tau_k}^{\tilde{h}}(v,v)|\r]  \notag\\
& \leq 
\E \l[\e^{K^-t \wedge \tau_k+\sigma^- l_{t\wedge \tau_k}}\r]^{1/2}\E\l[\e^{-K^-{t\wedge \tau_k}-\sigma^- l_{t\wedge \tau_k}}|\xi_{t\wedge \tau_k}^{(1)}|^2\r]^{1/2}+\frac{1}{2}\E\l[|\xi_{t\wedge \tau_k}^{(2)}|\r] + \frac 12  \E \l[|\xi_{t\wedge \tau_k}^{(3)}|\r] \notag \\
& \leq 
\alpha  \E \l[\e^{K^-t+\sigma^- l_t}\r]^{1/2}\l[\int_0^t \E[ \e^{-K^-s-\sigma^-  l_s}] \tilde{h}_s^2\, \vd s\r]^{1/2} +\frac{\beta}{2}\E\l[\e^{\frac{K^-}{2}t+\frac{\sigma^-}{2} l_t} \int_0^t \e^{\frac{K^-s}{2}+\frac{\sigma^- l_s}{2}}\tilde{h}_s\,\vd s \r] \notag \\
&\quad + \frac \gamma 2  \E \l[\e^{\frac{K^-}{2}t+\frac{\sigma^-}{2} l_t}\int_0^t\e^{\frac{K^-}{2}s+\frac{\sigma^-}{2} l_s}\tilde{h}_s\,\vd l_s \r]
\end{align}
where $\xi^{(1)}$, $\xi^{(2)}$ and $\xi^{(3)}$ are defined as above. Letting
$k$ tend to $\infty$ then yields
\begin{align*}
&\E \l[|W_t^{\tilde{h}}(v,v)|\r] <\infty.
\end{align*}
Recall that  $|\nabla P{\bf.}f|$ is bounded on $[0, T]\times M$ and $|\nabla \bd P{\bf.}f|$ is bounded on 
$[\eps, T]\times M$ for $0<\eps<T$, see Remark \ref{Remark_eps}.
Hence we may again the claimed formulas show first for $f$ replaced by $f_\eps:=P_\eps f$ and then take the
limit as $\eps\downarrow0$ in the final formulas. Hence we may assume that $|\nabla P{\bf.}f|$ and $|\nabla \bd P{\bf.}f|$ are bounded on 
$[0, T]\times M$, so that \eqref{local-M1} is a true martingale. By taking expectations and passing to the limit
as $\eps\downarrow0$, 
inequality \eqref{Hessian-formula2} is obtained.


Let now  $\tilde{h}_s=\frac{T-s}{T}$ for $s\in [0,T]$. 
It is straightforward to deduce from  \eqref{Hessian-formula2} and \eqref{W-2}  that
\begin{align*}
|\Hess_{P_Tf}|& \leq \|\nabla f\|_{\infty}\e^{-KT/2}\E[\e^{\sigma l_T/2}]\E\l[\int_0^T |\tilde{Q}_s|^2h(s)^2\, \vd s\r] ^{1/2} +\|\nabla f\|_{\infty} \E\l[|W_T^{\tilde{h}}(v,v)|\r]\\
&\leq  \|\nabla f\|_{\infty}\l(\alpha \sqrt{T}+\frac{\beta}{2} T+\frac{1}{\sqrt{T}} \r) \E\l[\e^{\sigma l_T}\r] \e^{K^-T} +\frac{\gamma}{2}\|\nabla f\|_{\infty}  \E \l(\e^{\frac{1}{2}\sigma l_T}\int_0^T\e^{\frac{1}{2}\sigma l_s}\, \vd l_s\r)\e^{K^-T}
\end{align*}
which shows the second claim.
\end{proof}

For manifolds with specific boundary properties,
more refined results can be derived.

\begin{corollary}\label{cor2}
Assume that the boundary $\partial M$ is empty or $-\nabla N \geq 0$ and
 $\nabla ^2N+R(N)=0$.
 Moreover,  suppose that $\Ric_V\geq K>0$, $\alpha:=\|R\|_{\infty}<\infty$ and
$\beta:=\|\nabla \Ric_V^{\sharp}+\vd^*R+R(\nabla V)\|_{\infty}<\infty.$
\begin{enumerate}[\rm(i)]
\item If $ \Ric_V\geq K, \ \nabla N \leq 0$, then for $f\in C_b^1(M)$,
  \begin{align*}
    |\Hess _{P_tf}|\leq \l(\frac{1}{\sqrt{\int_0^t\e^{Kr}\, \vd r}}+\frac{\alpha}{\sqrt{K}}+\frac{\beta}{K}\r) \e^{-{Kt}/{2}}(P_t|\nabla f|^2)^{1/2}.
  \end{align*}
\item If $\Ric_V=K$, $\nabla N=0, $ then for $f\in C_b^1(M)$,
  \begin{align*}
    |\Hess _{P_tf}|_{\HS}\leq   \l(\frac{1}{\sqrt{\int_0^t\e^{Kr}\, \vd r}}+\frac{n\alpha}{\sqrt{K}}+\frac{n\beta}{K}\r) \e^{-{Kt}/{2}}(P_t|\nabla f|^2)^{1/2}.
  \end{align*}
\end{enumerate}

\end{corollary}

\begin{proof}
If $\nabla N\leq 0$ and $\nabla ^2N+R(N)=0$, then 
$\II\geq 0$, which together with  $\Ric_Z\geq K$,
implies by \cite[Corollary 3.2.6]{Wbook14} that
$|\nabla P{\bf.}f|$ is bounded on $[0,t]\times M$.
Choosing $\tilde{h}$ such that $|\tilde{h}|\leq 1$, we have
\begin{align*}
\l(\E |W_t^{\tilde{h}}(v,v)|^2\r)^{1/2}&\leq \alpha \e^{-Kt/2}\l(\int_0^t \e^{-Ks}\, \vd 
s\r)^{1/2}+\frac{\beta}{2}
\e^{-Kt/2}\int_0^t \e^{-Ks/2}\, \vd s\\
&\leq \frac{\alpha}{\sqrt{K}}\e^{-{Kt}/{2}}+\frac{\beta}{K} \e^{-{Kt}/2}.
\end{align*}
Combining this with Theorem \ref{Them-hessian-gradient},  we conclude that 
\begin{align*}
|\Hess_{P_tf}| \leq (P_t|\nabla f|^2)^{1/2}\e^{-Kt/2}\l(\int_0^t\e^{-Ks}h^2(s)\, \vd s\r)^{1/2}+(P_t|\nabla f|^2)^{1/2}\l(\frac{\alpha}{\sqrt{K}}+\frac{\beta}{K} \r)\e^{-{Kt}/{2}}.
\end{align*}
The following choice of $h$:
\begin{equation}\label{eqn-k}
h(s)=-\frac{\e^{Ks}}{\int_0^t\e^{Kr}\, \vd r}, \qquad s\in [0,t],
\end{equation}
then leads to the first inequality.\smallskip

If $\nabla N=0$ and $\Ric_V=K$, then $\tilde{Q}_t=\e^{-{Kt}/{2}}\ptr_{t}$ and
 \begin{align*}
 &   \Hess_{P_tf}(v,v)\\
 &=\e^{-{Kt}/{2}}\E\l[-\bd f(\ptr_t v)\int_0^t\e^{-{Ks}/{2}}h(s)\langle   v,
    \vd  B_s\rangle+\bd f\l(\ptr_t\int_0^t\ptr_s^{-1}\e^{-{Ks}/{2}}R(\ptr_s\vd B_s,\ptr_s (\tilde{h}(s)v))(\ptr_s v)\r)\r]
    \\
    &\quad -\frac{1}{2}\e^{-{Kt}/{2}}\E\l[\bd f\l(\ptr_t\int_0^t \e^{-{Ks}/{2}}\ptr_s^{-1}(\bd ^*R-R(Z)+\nabla \Ric_Z)^{\sharp}(\ptr_s(\tilde{h}(s)v), \ptr_s v)\, \vd s \r)\r].
  \end{align*}
This implies
\begin{align*}
|\Hess_{P_tf}|_{\HS}&\leq \e^{-{Kt}/{2}} (P_t|\nabla f|^2)^{1/2}\l(\int_0^t \e^{-Ks}h(s)^2\, \vd s\r)^{1/2}+n\alpha (P_t|\nabla f|^2)^{1/2} \l(\int_0^t \e^{-Ks}\, \vd s\r)^{1/2}\\
&\quad+\frac{n\beta}{2}\e^{-{Kt}/{2}}(P_t|\nabla f|^2)^{1/2}\int_0^t\e^{-{Ks}/{2}}\,\vd s.
\end{align*}
Choosing $h$ as in \eqref{eqn-k} then yields item (ii).
\end{proof}

\medskip

\section{Stein method and log-Sobolev inequality}\label{Stein-log-Sobolev}
In this section, we consider  $L=\Delta-\nabla V$ for $V\in C^2(M)$ such that
$$\mu(\vd x)=\e^{-V(x)}{\mbox{vol}}(\vd x)$$
is a probability measure where $\mbox{vol}(\vd x)$ denotes the volume measure on $M$.
Let $P_t=\e^{\frac12 Lt}$ be the contraction 
semigroup generated by $L$ on $L^2(\mu)$ with Neumann boundary conditions.
In \cite{CTW}, we used the Hessian formula to establish an HSI inequality on manifolds
without boundary, which contains
the new quantity called Stein discrepancy and in a certain sense  
improves the  classical log-Sobolev inequality.

To establish such kind of log-Sobolev inequalities on manifolds with boundary,
we first adapt
the definition of Stein kernel and Stein discrepancy to manifolds with boundary.
A symmetric 2-tensor $\tau_{\nu}\colon M \rightarrow T^*M\times T^*M $ on $M$ is said to be a Stein kernel for
a probability measure $\nu$ on $M$ if $\tau_{\nu}(v, w)\in L^1(\nu)$
for every $v,w\in T_xM$, $x\in M$, and
\begin{align}\label{tau-nu}
\int \langle \nabla V, \nabla f \rangle \,\vd\nu=\int \langle \tau_{\nu}, \Hess f\rangle_{\HS} \,\vd\nu,\quad f\in \mathcal{C}^\infty_N(L),
\end{align}
where $\nabla V$ is the first order part of the operator $L$
and where $$\mathcal{C}^\infty_N(L)=\big\{f\in C^\infty(M)\colon Nf|_{\partial M}=0,\ Lf\in\B_b(M)\big\}.$$
Since
$$\int Lf \,\vd \mu=\int_{\partial M} N(f)\,\vd \mu=0$$
for $f\in \mathcal{C}^\infty_N(L)$, it is easy to see that the identity map ${\id}$ is a Stein kernel for $\mu$.
\begin{definition}\label{def-Stein-kernel}
 Let $\tau_{\nu}$ be a Stein kernel for $\nu$. The Stein discrepancy is defined as
$$S(\nu\,|\,\mu)^2=\inf\int_M |\tau_{\nu}-\id|_{\HS}^2 \,\vd\nu,$$
 where the infimum is taken over all Stein kernels of $\nu$, and takes the value $+\infty$ if no Stein kernel exists.
\end{definition}

Let us first recall the classical log-Sobolev inequality on Riemannian manifolds when the boundary is convex.
Assume that
\begin{align*}
\Ric_V:=\Ric+\Hess_ V\geq K,\quad \II\geq 0
\end{align*}
holds for some positive constant $K$. Then the classical logarithmic Sobolev
inequality with respect to the measure $\mu$ indicates that for every
probability measure $\vd \nu=h \,\vd\mu$ with smooth density $h\colon M\rightarrow \R_+$,
\begin{align}\label{HPtf-If}
H(\nu\,|\,\mu)\leq \frac{1}{2K} I(\nu\,|\,\mu),
\end{align}
where
\begin{align*}
H(\nu\,|\,\mu)=\int h\log h\, \vd\mu=\Ent_{\mu}(h)
\end{align*}
is the relative entropy of $\vd\nu=h\,\vd \mu$ with respect to $\mu$ and
\begin{align*}
I(\nu\,|\, \mu)=\int \frac{|\nabla h|^2}{h}\, \vd \mu =I_{\mu}(h)
\end{align*}
the Fisher information of $\nu$ (or $h$) with respect to $\mu$. This result is known as the Bakry-\'Emery criterion due to \cite{BE} for the logarithmic Sobolev
inequality.  Let us recall the following observations. 

\begin{lemma}\label{HI-NP}
Assume that
$$\Ric_V\geq K\ \mbox{and} \quad \II\geq 0$$
for some positive constant $K$, and let $\tau_{\nu}$
be a Stein kernel for $\vd \nu=h\,\vd \mu$ where $h\in C_0^{\infty}(M)$.
For $t>0$ let $\vd \nu^t=P_th\,\vd\mu$. Then
\begin{itemize}
  \item [(i)] (Integrated de Bruijn's formula)
  \begin{align*}
  H(\nu\,|\,\mu)=\Ent_{\mu}(h)=\frac12 \int_0^{\infty}I_{\mu}(P_th)\, \vd t;
  \end{align*}
  \item [(ii)] (Exponential decay of Fisher information) 
  \begin{align*}
  I_{\mu}(P_th)=I(\nu^t\,|\, \mu)\leq \e^{-Kt}I(\nu\,|\, \mu)=\e^{-Kt}I_{\mu}(h),\quad t\geq0.
  \end{align*}
\end{itemize}
\end{lemma}

\begin{proof}
Since $N(P_tf\log P_tf)=0$ and $\l(\frac12 L-\frac{\partial}{\partial t}\r)(P_th\log P_th)=\frac{|\nabla P_th|^2}{2P_th}$, we have
\begin{align*}
H(\nu\,|\,\mu)&=\int_M h\log h\,\vd\mu=-\int_M\int_0^{\infty}\frac{\vd\l( P_th\log P_th\r)}{\vd t}\, \vd t\,\vd\mu \\
&=\int_0^{\infty}\l(\int_M\Big(\frac12 L-\frac{\partial}{\partial t}\Big)(P_th\log P_th) \,\vd \mu \r) \vd t \\
&=\frac12 \int_0^{\infty}\int_M\frac{|\nabla P_th|^2}{P_th}\,\vd \mu\, \vd t.
\end{align*}
The second assertion can be checked by observing first from the derivative formula that
\begin{align*}
  |\nabla P_th|^2
  &\leq \e^{-Kt}\l(P_t\big|\nabla (\sqrt{h} \,)^2\big|\r)^2\leq
   4\e^{-Kt} (P_th)\, P_t\big|\nabla \sqrt{h}\big|^2
\end{align*}
which implies 
\begin{align*}
  I_{\mu}(P_th)&=\int_M \frac{|\nabla P_th|^2}{P_th}\,\vd \mu \leq 4\int_M\e^{-Kt} \frac{(P_th) P_t|\nabla \sqrt{h}|^2}{P_th}\,\vd \mu\\
               &=4\int_M\e^{-Kt} P_t|\nabla \sqrt{h}|^2\,\vd \mu =4\e^{-Kt} \int_M |\nabla \sqrt{h}|^2\,\vd \mu\\
               &=\e^{-Kt}I(\nu\,|\, \mu)=\e^{-Kt}I_{\mu}(h).\qedhere
\end{align*}
\end{proof}

All expressions should be considered for $h+\eps$ as $\eps\downarrow0$.
We continue our discussion under the condition that $\nabla^2N+R(N)=0$
and $-\nabla N\geq 0$. The following assertions describe the
relationship between the relative entropy and Stein discrepancy.

\begin{lemma}\label{esti-S-th2}
 Assume that $\alpha:=\|R\|_{\infty}<\infty$, $\beta:=\|\nabla \Ric_V^{\sharp}+\vd^*R+R(\nabla V)\|_{\infty}<\infty$ and
 $$ \nabla^2N+R(N)=0.$$
  Let $\vd \nu=h\, \vd \mu$  for $h\in C_0^{\infty}(M)$. 
 \begin{itemize}
   \item [(i)] If
   $
\Ric_V\geq K, \  \nabla N \leq 0
$,
then
\begin{align*}
I_{\mu}(P_th)\leq n^2\l(\frac{1}{\sqrt{\int_0^t\e^{Kr}\, \vd r}}+\frac{\alpha}{\sqrt{K}}+\frac{\beta}{K}\r)^2\e^{-Kt} S^2(\nu\,|\, \mu).
\end{align*}
   \item [(ii)] If
   $
\Ric_V=K,\  \nabla N=0,
$
then
\begin{align*}
I_{\mu}(P_th)\leq \l(\frac{1}{\sqrt{\int_0^t\e^{Kr}\, \vd r}}+\frac{n\alpha}{\sqrt{K}}+\frac{n\beta}{K}\r)^2 \e^{-Kt}S^2(\nu\,|\, \mu).
\end{align*}
 \end{itemize}
\end{lemma}
\begin{proof}
Let $g_t=\log P_th$.
 By the symmetry of $(P_t)_{t\geq 0}$ in $L^2(\mu)$,
\begin{align*}
I_{\mu}(P_th)=-\int (Lg_t)P_th\, \vd\mu=-\int (LP_tg_t)h\, \vd\mu=-\int LP_tg_t\, \vd\nu.
\end{align*}
Hence according to the definition of Stein kernel  and since $P_tg_t\in \mathcal{C}^\infty_N(L)$, we have
\begin{align*}
I_{\mu}(P_th)&=-\int \langle \id, \Hess _{P_tg_t} \rangle_{\HS}\, \vd\nu -\int \langle \nabla V, \nabla P_tg_t \rangle\, \vd\nu\\
&=\int \langle \tau_{\nu}-\id, \Hess _{P_tg_t} \rangle_{\HS} \, \vd\nu.
\end{align*}
This argument is due to \cite{LNP15} and connects the Fisher information
to the Stein discrepancy. 
We now first prove assertion (i). By the Cauchy-Schwartz inequality,
\begin{align*}
I_{\mu}(P_th)&=\int \langle \tau_{\nu}-\id, \Hess _{P_tg_t} \rangle_{\HS} \, \vd\nu\\
&\leq \l(\int |\tau_{\nu}-\id|_{\HS}^2\, \vd\nu\r)^{1/2}\l(\int |\Hess _{P_tg_t}|^2_{\HS}\, \vd\nu\r)^{1/2}\\
&\leq n\l(\int |\tau_{\nu}-\id|_{\HS}^2\, \vd\nu\r)^{1/2} \l(\frac{1}{\sqrt{\int_0^t\e^{Kr}\, \vd r}}+\frac{\alpha}{\sqrt{K}}+\frac{\beta}{K}\r)\e^{-\frac{K}{2}t}\l(\int P_t|\nabla g_t|^2\ \vd\nu\r)^{1/2}
\end{align*}
where Corollary \ref{cor2} is used for the function $g_t=\log P_th$.
Since
\begin{align*}
\int P_t|\nabla g_t|^2\, \vd \nu&=\int P_t|\nabla g_t|^2 h\,\vd\mu=\int |\nabla g_t|^2 P_th\, \vd\mu=\int \frac{|\nabla P_th|^2}{P_th}\,\vd\mu= I_{\mu} (P_th),
\end{align*}
it then follows that
\begin{align*}
I_{\mu} (P_th)\leq n^2 \l(\frac{1}{\sqrt{\int_0^t\e^{Kr}\, \vd r}}+\frac{\alpha}{\sqrt{K}}+\frac{\beta}{K}\r)^2\e^{-Kt} \int |\tau_{\nu}-\id|_{\HS}^2\, \vd\nu.
\end{align*}
Taking the infimum  over all Stein kernels of $\nu$, we finish the proof of (i).
Along the same steps, item (ii) can be
proved  by means of Corollary \ref{cor2} as well.
\end{proof}

Using the lemmata above, we are now in position to establish the following result.

\begin{theorem}\label{K-R1-s2}
Assume that $\alpha:=\|R\|_{\infty}<\infty, \  \beta:=\|\nabla \Ric_V^{\sharp}+\vd^*R+R(\nabla V)\|_{\infty}<\infty$ and
$$  \nabla ^2N+R(N)=0.$$
Let $\vd \nu=h\, \vd \mu$ with $h\in C_0^{\infty}(M)$.
\begin{enumerate}[\rm(i)] 
  \item If $\Ric_V\geq K, \  \nabla N \leq 0$, then
\begin{align*}
\quad H(\nu\,|\,\mu)
 &\leq\frac{1}{2K}\l(n^2(1+\eps)\l(\frac{\alpha}{\sqrt{K}}+\frac{\beta}{K}\r)^2S^2(\nu\,|\,\mu)\r)\wedge I(\nu\,|\, \mu)\\
&\quad-\frac{n^2}{2}\l(1+\frac{1}{\eps}\r)S^2(\nu\,|\,\mu)\ln\l(\frac{n^2(1+\frac{1}{\eps})KS^2(\nu\,|\,\mu) }{\Big(I(\nu\,|\,\mu)-n^2(1+\eps)\Big(\frac{\alpha}{\sqrt{K}}+\frac{\beta}{K}\Big)^2S^2(\nu\,|\,\mu)\Big)\vee 0+n^2(1+\frac{1}{\eps})K S^2(\nu\,|\,\mu)}\r)
\end{align*}
for every $\eps>0$. 
Moreover, if $\alpha=0$ and $\beta=0$, then
\begin{align*}
H(\nu\,|\, \mu)\leq  \frac{n^2}{2} S^2(\nu\,|\,\mu)\ln\l(1+\frac{I }{n^2KS^2(\nu\,|\,\mu)}\r).
\end{align*}
 \item If $\Ric_V= K, \  \nabla N = 0$, then
\begin{align*}
\quad H(\nu\,|\,\mu)
&\leq \frac{1}{2K}\l(n^2(1+\eps)\l(\frac{\alpha}{\sqrt{K}}+\frac{\beta}{K}\r)^2S^2(\nu\,|\,\mu)\r)\wedge I(\nu\,|\,\mu)\\
&\quad -\frac{1}{2}\l(1+\frac{1}{\eps}\r)S^2(\nu\,|\,\mu)\ln\l(\frac{(1+\frac{1}{\eps})KS^2(\nu\,|\,\mu) }{\Big(I(\nu\,|\,\mu)-n^2(1+\eps)\Big(\frac{\alpha}{\sqrt{K}}+\frac{\beta}{K}\Big)^2S^2(\nu\,|\,\mu)\Big)\vee 0+(1+\frac{1}{\eps})K S^2(\nu\,|\,\mu)}\r)
\end{align*}
for every $\eps>0$. 
 Moreover, if $\alpha=\beta=0$, then 
\begin{align*}
H(\nu\,|\, \mu)\leq \frac{1}{2} S^2(\nu\,|\,\mu)\ln\l(1+\frac{I }{KS^2(\nu\,|\,\mu)}\r).
\end{align*}
\end{enumerate}
\end{theorem}

\begin{proof}
We only need to prove the first estimate. To this end, we write $I=I(\nu\,|\,\mu)$ and $S=S(\nu\,|\,\mu)$ for simplicity.
By  Theorem \ref{esti-S-th2} and Lemma \ref{HI-NP}, we have
\begin{align*}
H(\nu\,|\, \mu)&\leq \frac12 \inf_{u>0}\l\{A\int_0^u\e^{-Kt}\,\vd t+B\int_u^{\infty}\frac{K}{\e^{Kt}(\e^{K t}-1)}\,\vd t+C\int_u^{\infty}\e^{-Kt}\, \vd t\r\}\\
&=\frac12 \inf_{u>0}\l\{\frac{A(1-\e^{-Ku})+C\e^{-Ku}}{K}+B\int_0^{\e^{-\alpha u}}\frac{r}{1-r}\,\vd r\r\}
\end{align*}
where 
\begin{align*}
&A=I(\nu\,|\, \mu);\ \  B=n^2\l(1+\frac{1}{\eps}\r)S^2(\nu\,|\, \mu);\\
&C=n^2(1+\eps)\l(\frac{\alpha}{\sqrt{K}}+\frac{\beta}{K}\r)^2S^2(\nu\,|\, \mu).
\end{align*}
It is easy to see that if $A\leq C$, then  $\inf$ is reached when $u$ tends to $\infty$; if 
$A>C$, then  $\inf$ is reached  for $\e^{\alpha u}=\frac{A-C+BK}{A-C}$ so
that
\begin{align*}
H(\nu\,|\,\mu)\leq \frac{C}{2K}+\frac{1}{2}B\ln \l(1+\frac{A-C}{BK}\r).
\end{align*}
We conclude that 
\begin{align*}
H(\nu\,|\,\mu)\leq \frac{C\wedge A}{2K}+\frac{1}{2}B\ln \l(1+\frac{(A-C)\vee 0}{BK}\r).
\end{align*}
The rest of the proof is the same replacing $B$ by
\begin{align*}
 &\l(1+\frac{1}{\eps}\r)S^2(\nu\,|\,\mu).
 \end{align*} The details are omitted here.
\end{proof}

Let $(M,g)$ be a connected complete Riemannian manifold $M$. Considering the specific case that $\Hess _V= K>0$,
then by Obata's Rigidity Theorem (see \cite[Theorem 2]{Tashiro:1965} or \cite[Theorem 3.4]{WY2014}), $M$ is
 isometric to $\mathbb{R}^n$. The following corollary shows that the result is
 consistent with Ledoux-Nourdin-Peccati \cite{LNP15} for the Gaussian measure on the Euclidean space $\R^n$.

\begin{corollary}\label{cor-Ricci-flat}
Let $(M,g)$ be a connected complete Riemannian manifold with boundary.  Assume that $\Hess _V= K>0$, $\nabla N=0$, and $\nabla^2N=0$. 
  Let $\vd \nu=h\, \vd \mu$  with $h\in C_0^{\infty}(M)$. Then,
\begin{align*}
I(\nu\,|\,\mu)\leq \frac12 S^2(\nu\,|\,\mu)\log \l(1+\frac{I(\nu\,|\,\mu)}{KS^2(\nu\,|\,\mu)}\r).
\end{align*}
\end{corollary}
\begin{proof}
From the condition $\Hess _V= K>0$, we know that the manifold  $M$ is
 isometric to $\mathbb{R}^n$, i.e.
$\|R\|_{\infty}=0,$ $\nabla \Ric_Z=0$ and $\bd^*R=0$.
Then by Theorem \ref{esti-S-th2} (ii),
\begin{align*}
I_{\mu}(P_th)\leq \frac{K}{\e^{2Kt}-\e^{Kt}} S^2(\nu\,|\, \mu).
\end{align*}
The assertion can be obtained  by a same arguments as in the proof of Theorem \ref{K-R1-s2}.
\end{proof}

\subsection*{Conflict of Interest and Ethics Statements}
The authors declare that there is no conflict of interest. Data sharing is not applicable to this article as no data-sets were created or analyzed in this study.

\bibliographystyle{amsplain}%

\bibliography{Hessian-boundary-26}

\end{document}